\documentclass[10pt,a4paper,reqno]{amsart}
\usepackage{amsfonts}
\usepackage{amssymb}
\usepackage{amsmath}
\usepackage{amsthm}
\usepackage{cite}
\usepackage{enumitem}
\usepackage{tikz}
\usepackage{subfigure}
\usepackage{stackrel}
\usepackage{hyperref}
\usepackage{dsfont}
\usepackage{caption}
\usepackage{url}
\theoremstyle{plain}
\newtheorem{thm}{Theorem}[section] 
\newtheorem{defn}[thm]{Definition}
\newtheorem{lem}[thm]{Lemma}
\newtheorem{prop}[thm]{Proposition}
\newtheorem{cor}[thm]{Corollary}
\newtheorem{rem}[thm]{Remark}

\newcommand{\plus}{\,+\,}
\providecommand{\ind}{\mathds{1}} 
\providecommand{\sm}{\setminus} 
\providecommand{\N}{\mathbb{N}}
\providecommand{\R}{\mathbb{R}}
\providecommand{\T}{\mathbb{T}}  
\providecommand{\Z}{\mathbb{Z}}

\providecommand{\bS}{\mathbb{S}}  
\providecommand{\C}{\mathbb{C}}
\providecommand{\cA}{\mathcal{A}}

\providecommand{\cL}{\mathcal{L}}
\providecommand{\cM}{\mathcal{M}}
\providecommand{\cNP}{\mathcal{N}_{\mathcal{P}}}

\providecommand{\eps}{\varepsilon}
\providecommand{\ov}{\overline}
\providecommand{\un}{\underline}

\providecommand{\wto}{\rightharpoonup}
\providecommand{\weak}{\rightharpoonup}
\providecommand{\skp}[2]{\langle#1,#2\rangle}
\providecommand{\les}{\lesssim}
\providecommand{\ges}{\gtrsim}

\DeclareMathOperator{\supp}{supp}

\DeclareMathOperator{\id}{id}

\begin{document}

\allowdisplaybreaks

\title[Ground state solutions to nonlinear wave equations]{Ground state solutions to generalized nonlinear wave equations with infinite-dimensional kernel}

\author{Rainer Mandel\textsuperscript{1}}
\author{Tobias Weth\textsuperscript{2}}
\address{\textsuperscript{1}Karlsruhe
Institute of Technology, Institute for Analysis, Englerstra{\ss}e 2, 76131 Karlsruhe, Germany}
\address{\textsuperscript{2}Institut f\"ur Mathematik, Goethe-Universit\"at Frankfurt, 60629 Frankfurt am
Main}
  

\keywords{}
\date{\today}   

\begin{abstract}
  The present paper is devoted to existence results for time-periodic solutions of generalized nonlinear wave
  equations in a closed Riemannian manifold $M$. Our main focus lies on the doubly degenerate setting where
  the associated generalized wave operator has an infinite dimensional kernel and the nonlinearity may vanish
  on open subsets of $M$. To deal with this setting, we apply a direct variational approach based on a new
  variant of the nonlinear saddle point reduction to the associated Nehari-Pankov set.
  This allows us to find ground state solutions and to characterize the associated ground state energy by a
  fairly simple minimax principle.
\end{abstract}

\maketitle
\allowdisplaybreaks 
\setlength{\parindent}{0cm}

\section{Introduction}
\label{sec:introduction}
The present paper is devoted to a class of generalized nonlinear wave equations of the form
\begin{equation}
  \label{eq:nonlinear-wave-intro}
\cA u + \partial_{tt} u = q(x,t)|u|^{p-2}u,\qquad x \in M,\quad t \in \bS^1,  
\end{equation}
where $\bS^1$ is the unit circle and $(M,g)$ is a closed Riemannian manifold. 

We mostly consider differential operators of the form $\cA = P(-\Delta)$ where $-\Delta = -\Delta_g$ is
the Laplace-Beltrami operator on $(M,g)$ and $P$ is a polynomial $P$ with real coefficients and $P(\tau) \to \infty$ as $\tau \to \infty$. By assuming $t \in \bS^1$, we restrict our attention to solutions of
(\ref{eq:nonlinear-wave-intro}) which are $2\pi$-periodic in time.  Moreover, we assume that $q \in L^\infty(M \times \mathbb{S}^1)$ and $p>2$. Equations of the form~(\ref{eq:nonlinear-wave-intro}) have received extensive attention in recent years, and they are known under different names depending on the specific form of the generalized wave operator $\cA+ \partial_{tt}$. In the case $\cA=-\Delta$, we are dealing with a class of superlinear wave equations, while (\ref{eq:nonlinear-wave-intro})  is called a nonlinear Klein-Gordon equation with mass $c> 0$ in the case where $\cA=-\Delta + c$. In the case $\cA = \Delta^2$, equation (\ref{eq:nonlinear-wave-intro}) is called a nonlinear beam equation 
if $M= \bS^1$ and a {nonlinear} membrane equation if $M$ is two-dimensional, see e.g.
\cite{BaDi_PeriodicSolutions} and the references therein.

From a functional analytic point of view, it is convenient to consider (\ref{eq:nonlinear-wave-intro}) in a more general framework of metric measure spaces. So in the following we assume that $(M,d,\mu)$ is a metric measure space, i.e., $(M,d)$ is a compact metric space and $\mu$ is a Borel measure
on $M$. Furthermore, we assume that the operator $\cA: D(\cA) \subset L^2(M) \to L^2(M)$ in (\ref{eq:nonlinear-wave-intro}) satisfies the following assumptions:
\begin{itemize}
  \item[(A)] $\cA$ is a selfadjoint operator with compact resolvent and
    \begin{itemize}
    \item[(A1)] $\inf \limits_{\phi \in D(\cA) \setminus \{0\}}\frac{\skp{\mathcal A
    \phi}{\phi}}{\skp{\phi}{\phi}}>-\infty$,
    \item[(A2)] $\sup \bigl\{ \frac{\skp{\mathcal A \phi}{\phi}}{\skp{\phi}{\phi}}\::\: \phi \in D(\cA)
    \setminus \{0\},\; \phi \equiv 0 \; \text{in }M \setminus U\bigr\} = \infty\:$ for any open subset
    $\varnothing \not= U \subset M$,
    \item[(A3)] $\phi \equiv 0 \text{ in } U$ implies $\cA\phi\equiv 0$ in $U$ 
    for any open subset $\varnothing \not= U \subset M$.
    \end{itemize}
\end{itemize}
We stress that $(A)$ is satisfied in the setting specified above, i.e., when $(M,g)$ is a closed Riemannian
manifold with induced measure $\mu_g$ and $\cA = P(-\Delta)$ with a real polynomial with $P(\tau) \to \infty$ as $\tau \to \infty$.

As before we let $q \in L^\infty(M \times \mathbb{S}^1)$, where, here and in the following, the spaces $L^p(M\times \mathbb{S}^1)$,
$1 \le p \le \infty$ are defined with respect to the product measure of $\mu$ with the one-dimensional 
Lebesgue measure on $\mathbb{S}^1$.

As a consequence of assumption (A), the spectrum of $\cA$ consists of an unbounded
sequence of eigenvalues 
$$
\nu_0(\cA) \le \nu_1(\cA) \le \dots \le \nu_k(\cA) \le \dots \to +\infty
$$
(counted with multiplicity), and there exists an orthonormal basis of $L^2(M)$ of associated eigenfunctions
$\zeta_k\in D(\cA)$, $k \in \N_0:= \N\cup\{0\}$. As shown in  Proposition~\ref{prop-self-adjoint-realization}
below, the generalized wave operator $L_\cA:= \cA + \partial_{tt}$ then has a self adjoint realization in
the Hilbert space $L^2(M \times \mathbb{S}^1)$ with spectrum 
$$
  \sigma(L_\cA) = \{\nu_k(\cA)- l^2 : l,k \in \N_0\}.  
$$
The corresponding eigenfunctions are of the form
$$
(x,t) \mapsto (a_1 \cos(lt)+ a_2 \sin(lt))\zeta_k(x),
$$
where $a_1, a_2 \in \R$, and $\zeta_k$ is an eigenfunction of $\cA$ on $M$ corresponding to the eigenvalue
$\nu_k(\cA)$. Hence we observe the following:
 \begin{equation}
   \label{eq:square-observatoin}
   \begin{aligned}
  &\textit{If $\nu_k(\cA)$ is a square for infinitely many $k \in \N_0$,}\\
 &\textit{then  the kernel $\ker(L_\cA)$ of $L_\cA$ is
infinite dimensional.}     
   \end{aligned}
\end{equation}
While (\ref{eq:nonlinear-wave-intro}) has the structure of a strongly indefinite variational problem and therefore allows the application of tools from critical point theory, the presence of an infinite-dimensional kernel of the generalized wave operator $L_\cA$ leads to severe technical and conceptual difficulties and is therefore excluded by assumption in most of the available literature. To discuss these difficulties in an abstract framework, let us consider a real Hilbert space $(E,\|\cdot\|)$ with an orthogonal splitting $E= E^+ \oplus E^0  \oplus E^-$, so every $u \in E$ writes in a unique way as $u = u^+ + u^0 + u^-$ with $u^\pm \in E^\pm$ and $u^0 \in E^0$. Under suitable additional assumptions, (\ref{eq:nonlinear-wave-intro}) can be written as the Euler Lagrange equation of a functional $\Phi \in C^1(E,\R)$ of the form 
\begin{equation}
  \label{eq:general-functional-intro}
\Phi(u) = \|u^+\|^2-\|u^-\|^2 - I(u),
\end{equation}
where $I \in C^1(E, \R)$ is superquadratic in a sense to be specified below. Hence solutions of (\ref{eq:nonlinear-wave-intro}) can be found as critical points of $\Phi$. Of particular
interest are {\em ground state solutions}, i.e., $\Phi$-minimizers within the set of non-zero critical points of $\Phi$.  
A key role in the search of these least energy nonzero critical points of $\Phi$ is played by the associated {\em Nehari-Pankov set}\footnote{This set is also called the Nehari-Pankov manifold, but additional restrictions are needed to guarantee that $\cNP$ is a manifold, see \cite{pankov-2005}.}
$$  
\cNP:= \{u \in E \setminus (E^{-}\oplus E^0) \::\: \Phi'(u)\Big|_{E^- \oplus E^0 \oplus \R u}= 0\}. 
$$
which contains all critical points of $\Phi$. Indeed, in the case where $\dim E^0< \infty$, the following properties have been shown in \cite{szulkin-weth-2009,SzuWet} for a fairly large class of nonlinear potentials $I$:
\begin{itemize}
\item[(i)] The minimum of $\Phi$ on $\cNP$ is attained, and every minimizer is a least energy nonzero critical
point of $\Phi$;
\item[(ii)] for every $u \in E \setminus (E^- \oplus E^0)$, the half space $E^- \oplus E^0 + \R^+ u$ intersects $\cNP$ in precisely one point which is the unique global maximum of $\Phi\big|_{E^- \oplus E^0 + \R^+ u}$.
\end{itemize}
In the case where $E^0 = \{0\}$, property (i) has already been proved independently in \cite{pankov-2005}
in a setting related to periodic Schrödinger equations and in \cite{ramos-tavares-2008} in the framework of an elliptic Hamiltonian system. We point out that property (ii) gives rise to a nonlinear saddle point reduction of the functional $\Phi$ and implies, in particular, the relatively simple minimax characterization
$$
c = \inf_{w \in E^+ \setminus \{0\}}\; \sup_{u \in E^- \oplus E^0 + \R^+ w}\Phi(u)
$$
of the minimal nonzero critical value $c:= \min \Phi\big|_{\cNP}$, also known as ground state energy.  

In the case where $\dim E^0= \infty$, the properties (i) and (ii) above are much more difficult to obtain. In fact we are only aware of the work of Bartsch and Mederski \cite{bartsch-mederski-2015,bartsch-mederski-2017-JFA,bartsch-mederski-2017-JFPT} where a saddle point reduction to the Nehari-Pankov set $\cNP$ has been developed and applied in the context of nonlinear Maxwell equations.
As can be seen from their work, drastic changes are needed to deal with this indefinite and strongly degenerate variational
setting. In particular, one is forced to leave the Hilbert space framework, since one needs to replace $E^0$ by a dense subspace endowed with a non-Hilbertian norm adapted to the nonlinear potential $u \mapsto I(u)$. 

In the case where the weight function $q$ is bounded below by
a positive constant, we could in principle apply the abstract theory in \cite{bartsch-mederski-2017-JFPT} to the generalized wave equation (\ref{eq:nonlinear-wave-intro}). However, a different approach in a modified setting is needed to include the case where $q$ has a nontrivial zero set. In fact, as we shall discuss in Remark~\ref{comparison-remark}(ii)
below, a key abstract assumption in \cite{bartsch-mederski-2017-JFPT} is not satisfied in this doubly
degenerate case, which also prevents the application of a classic dual variational approach in the space
$L^p(M)$ as in \cite{BreCorNir}, see also \cite[Chapter I.6]{Struwe}.
In Section~\ref{sec:an-abstr-exist} below, we therefore
present such an alternative approach to a nonlinear saddle point reduction under a different set of
assumptions. This approach leads rather directly to the existence of ground states of $\Phi$.

To present our main results for (\ref{eq:nonlinear-wave-intro}), we need to cast this equation in the abstract variational framework outlined above. In the following, we put $H:= L^2(M \times \bS^1)$, we let
$P^\pm \in \cL(H)$ denote the spectral projections associated with the positive and negative part of the
spectrum $\sigma(L_\cA)$, and we let $P^0 \in \cL(H)$ be the orthogonal projection on the kernel $\ker(L_\cA)$ of $L_\cA$. So we have $P^+ + P^- + P^0 = \id_H$. To shorten the notation, we also set $u^\pm:= P^\pm u$ and $u^0:= P^0u$ for $u \in H$. 
We then consider the subspaces $E^\pm \subset H$ given by
 $$
 E^\pm := \Bigl \{u \in P^\pm H \::\: \|u\|_\pm < \infty \Bigr\},
 $$
 where, here and in the following,
 $$
 \|u\|_{+} := \Bigl(  \int_{(0,\infty)} \lambda \,d \langle E(\lambda)u,u \rangle_H  \Bigr)^{1/2},
 \qquad \|u\|_{-} := \Bigl(  \int_{(-\infty,0)} |\lambda|\,d \langle E(\lambda)u,u \rangle_H  \Bigr)^{1/2}
$$
and $\lambda \mapsto E(\lambda)$ denotes the spectral resolution associated with the operator $L_\cA$ in
$H$.
 It is easy to see that $(E^\pm,\|u\|_{\pm})$ are pre-Hilbert spaces with scalar products $\langle \cdot,
 \cdot\rangle_\pm$ given by 
 $$
 \langle u,v \rangle_+ := \int_{(0,\infty)} \lambda \,d\langle  E(\lambda)u,v \rangle_H, \qquad 
 \langle u,v \rangle_-= \int_{(-\infty,0)} |\lambda| \,d\langle E(\lambda)u,v \rangle_H .
 $$
Recalling that we consider a fixed exponent
 $p>2$ in (\ref{eq:nonlinear-wave-intro}), we now make the following assumptions:
 \begin{itemize}
 \item[$(CE)_p$] {\em (Compact embedding)} The spaces $(E^\pm,\|\cdot\|_\pm)$ are 
 compactly embedded into $L^p(M \times \mathbb{S}^1)$.
 \item[$(CC)_q$] ({\em $q-$control condition}) There exists a constant $C>0$ with
   $$
 \int_{M\times \mathbb{S}^1} |u|^2\,d(x,t)  
   \le C \int_{M \times \mathbb{S}^1} q(x,t)|u|^2\,d(x,t) \qquad \text{for all }u \in \ker(L_\cA).
   $$
 \end{itemize}
Here and in the following, it will be convenient to let $dx$ refer to integration with respect to the measure
$\mu$ on $M$ and the symbol $d(x,t)$ refer to integration with respect to the product measure on $M\times
\bS^1$.  Since $\mu(M)< \infty$ by our assumptions on $M$, condition $(CE)_p$ implies that the spaces $E^\pm$
are compactly embedded into $L^2(M \times \mathbb{S}^1)$ as well. In particular, it ensures that 
the set of {\em nonzero} eigenvalues of $L_\cA$ has no accumulation points in $\R$ and that 
all nonzero eigenvalues have finite multiplicity. From this we deduce that the spaces
$(E^\pm,\|\cdot\|_{\pm})$ are complete, so they are Hilbert spaces.  Moreover, condition $(CC)_q$ allows us to
introduce a topology on $\ker(L_\cA)$ adapted to the nonlinearity in (\ref{eq:nonlinear-wave-intro}). For
this we define, for any given $u\in H$, 
$$ 
  \|u\|_0 :=  \Bigl(\int_{M \times \mathbb{S}^1} q(x,t)|u^0|^p\,d(x,t)\Bigr)^{\frac{1}{p}}.
$$
We note that $(CC)_q$ implies that $\{u \in \ker(L_\cA) \::\: \|u\|_0<\infty\}$ is a subspace of
$\ker(L_\cA)$ and hence of $L^2(M\times\mathbb{S}^1)$, and that $\|\cdot\|_0$ is a norm on this space.
Moreover, $(CC)_q$ also implies that the completion of this space with respect to $\|\cdot\|_0$ is again a subspace of $\ker(L_\cA)$, and we denote this completion by $E^0$. Then 
$$ 
  E := E^-\oplus E^0\oplus E^+ 
   = \Bigl \{u \in H\::\: \|P^\pm u\|_\pm < \infty \quad \text{and}\quad \int_{M \times \mathbb{S}^1} q(x,t)|u|^p\,d(x,t) < \infty \Bigr\} 
$$
is a Banach space endowed with the norm $\|\cdot\|$ given by
 $$
  \|u\|^2 := \|P^+ u\|_+^2 + \|P^- u\|_-^2 + \|P^0u\|_{{0}}^2.
$$
We can now define the notion of weak solutions $u \in E$ of (\ref{eq:nonlinear-wave-intro}). 

\begin{defn}
  \label{def-weak-solution}
  A function $u \in E$ will be called a {\em weak solution of (\ref{eq:nonlinear-wave-intro}) on $M \times \mathbb{S}^1$} if
  \begin{equation*}
    \langle P^+ u , P^+ v \rangle_+ - \langle P^- u , P^- v \rangle_- 
    = \int_{M \times \mathbb{S}^1} q(x,t)|u|^{{p-2}}u v \,d(x,t) \quad \text{for every $v \in E$.}
  \end{equation*}
\end{defn}
Note that regular functions $u\in D(L_{\cA})\cap E$ satisfy
\begin{equation} \label{eq:SKP_vs_operator}
  \langle P^+u ,P^+v \rangle_+ -\langle P^-u,P^-v \rangle_{-}
  = \skp{L_{\cA}u}{v}_H
  = \int_{M\times\mathbb{S}^1} (L_{\cA}u)v\,d(x,t). 
\end{equation}
So weak solutions of (\ref{eq:nonlinear-wave-intro}) are critical points of the energy functional
$$
\Phi \in C^1(E,\R), \qquad \Phi(u) = \frac{1}{2}\Bigl(\|P^+u\|_+^2 - \|P^-u\|_-^2\Bigr) - \frac{1}{p} \int_{M \times \mathbb{S}^1} q(x,t)|u|^p\,d(x,t).
$$
As before, a $\Phi$-minimizer within the set of nontrivial weak solutions of (\ref{eq:nonlinear-wave-intro})
will be called a ground state solution of (\ref{eq:nonlinear-wave-intro}) on $M \times \mathbb{S}^1$.
Our main existence result in this setting is the following.

\begin{thm}
\label{thm-general-hyperbolic-intro}
  Let $\cA$ be a selfadjoint operator in $L^2(M)$ satisfying the assumption (A).
  Suppose that $p>2$ and $q \in L^\infty(M \times \mathbb{S}^1)$, $q \ge 0$ are chosen with the
  properties that $q>0$ on some open subset of $M \times \mathbb{S}^1$ and that conditions $(CE)_p$ and $(CC)_q$ hold.
  Then (\ref{eq:nonlinear-wave-intro}) admits a ground state solution on $M
  \times \mathbb{S}^1$. 
\end{thm}

We note here that it already follows from $(CC)_q$ that $q$ must be positive on a set of positive measure,
but not necessarily on an open set. Clearly, this holds if $q$ is continuous on $M
\times \bS^1$.
While our assumptions do not require $E^0$ to be infinite dimensional or even nontrivial, we wish to
concentrate on applications to the most difficult case where $\dim E^+ = \dim E^0 =\dim E^- = \infty$ in the following. To demonstrate
the applicability of Theorem~\ref{thm-general-hyperbolic-intro}, we discuss some specific examples. The
respective assumptions on the exponent $p$ and the coefficient function $q$ arise from the verification of
the contions $(CE)_p,(CC)_q$ that we will present later.
We first consider the case where $M= \mathbb{S}^1$, so we consider solutions of
(\ref{eq:nonlinear-wave-intro})  which are $2\pi$-periodic in time and space. Explicit computations for the
classical wave operator in one spatial dimension reveal the following.  

\begin{thm}  \label{thm-wave-1+1}
  Let $M= \mathbb{S}^1$, $\cA=-\Delta$ and $2<p<\infty$.  Moreover, suppose that $q \in L^\infty(\mathbb{S}^1 \times \mathbb{S}^1)$ is
  nonnegative and satisfies
    \begin{equation}  \label{eq:q-con-wave-1+1}
     \inf_{\Omega} q  >0 \qquad 
     \begin{aligned}
       &\text{for a subset $\Omega \subset \mathbb{S}^1 \times \mathbb{S}^1$ containing a rectangle }\\
&\text{$[a_1,b_1]\times [a_2,b_2]$ with 
       $b_1+b_2-a_1-a_2 > 2\pi.$}        
   \end{aligned}
 \end{equation}
  Then (\ref{eq:nonlinear-wave-intro}) admits a ground state solution on $\mathbb{S}^1 \times \mathbb{S}^1$.
  In particular, this holds for $\Omega = \omega\times \mathbb S^1$ or $\Omega = \mathbb S^1\times \omega$
  where $\omega\subset \mathbb S^1$ is an  arbitrary nonempty open set.
\end{thm}

\begin{thm}
\label{thm-S-1-intro}
  Let $M= \mathbb{S}^1$, $\cA= (-\Delta)^m$ 
  for some $m \in \N,m\geq 2$ and $2<p<\infty$. Moreover, suppose that $q \in L^\infty(\mathbb{S}^1 \times
  \mathbb{S}^1)$ is nonnegative and satisfies
    \begin{equation}  \label{eq:q-con-S1}
     \inf_{\Omega} q  >0 \qquad \text{for some nonempty open subset }\Omega \subset \mathbb{S}^1
   \times \mathbb{S}^1. 
 \end{equation} 
  Then (\ref{eq:nonlinear-wave-intro}) admits a ground state solution on $\mathbb{S}^1 \times \mathbb{S}^1$.
\end{thm}

In the above theorems case we have $\nu_k(\cA)= k^{2m}$ with $m\in\N$, so $\nu_k(\cA)$ is a square for all $k
\in \N_0$ and therefore $\dim E^0 = \infty$ by (\ref{eq:square-observatoin}). The assumptions
\eqref{eq:q-con-wave-1+1},\eqref{eq:q-con-S1} are needed for the verification of the abstract condition
$(CC)_q$ above. It is worth pointing out that the case $m \ge 2$ leads to a much weaker assumption. 

Next we consider the case $M=\T^N$, where $\T^N= \mathbb{S}^1 \times \dots \times \mathbb{S}^1$ is the flat
$N$-torus.  

\begin{thm}
\label{thm-T-N-intro}
  Let $M= \T^N$, let $\cA=(-\Delta)^m$ for some even $m \in \N$, and let $2 < p <\frac{2N}{(N-m)_+}$.
  Moreover, suppose that $q \in L^\infty(\T^N \times \mathbb{S}^1)$ is nonnegative and satisfies
    \begin{equation}
 \label{eq:q-con-TN}
  \begin{cases}
     \;\;\inf \limits_{\Omega} q  &>0 \qquad 
     \text{for some nonempty open subset $\Omega \subset \T^N \times \mathbb{S}^1$ in case $m = 2$;}\\      
     \inf \limits_{\omega \times \mathbb{S}^1} q &>0 \qquad \text{for some nonempty open subset $\omega
     \subset \T^N$ in case $m>2$.}
   \end{cases}
 \end{equation}
 Then (\ref{eq:nonlinear-wave-intro}) admits a ground
 state solution on $\T^N \times \mathbb{S}^1$.
\end{thm}
 
We note here that the set of eigenvalues of $-\Delta$ on $\T^N$ 
is given by $\{|k|^2\::\: k \in \Z^N\}$, and consequently the set $\{|k|^{2m}\::\: k \in \Z^N\}$ of
eigenvalues of $(-\Delta)^m$ on $\T^N$ contains an infinite number of squares, which by
(\ref{eq:square-observatoin}) implies that $\dim E^0 = \infty$. We stress that this is true for every $m \in
\N$, whereas the evenness assumption for $m$ in Theorem~\ref{thm-T-N-intro} allows us to prove 
$(CE)_p$ in a rather direct way, see Corollary~\ref{cor-compact-embedding-torus} below.

\begin{rem}  \label{rem:polyharmonicwaves_on_torus}
{\rm   Our method does not apply to the nonlinear wave equation on tori in higher space dimensions, which
  corresponds to  $N\geq 2$ and $m=1$ in Theorem~\ref{thm-T-N-intro}. 
  Indeed, the compactness property $(CE)_p$ fails for any $p\in
  [1,\infty]$. To see this, take any infinite sequence of tuples $(k,l)\in\Z^N\times\Z$ such that 
  $||k|^2-l^2|$ is bounded away from zero and infinity. As an example, one may choose $k=(l,1,0,\ldots,0)$.  
  A straightforward computation then shows that the corresponding eigenfunctions, say 
  $$
    (x,t)\mapsto \cos(k_1x_1)\cdot\ldots\cdot \cos(k_Nx_N)\cos(lt),
  $$ 
  are bounded in $E^+\oplus E^-$ and weakly convergent to zero  with constant $L^p$-norm. In particular, these
  eigenfunctions do not converge to zero in $L^p(\mathbb T^N\times\bS^1)$. So $E^+\oplus E^-$ does not embed
  compactly into $L^p(\mathbb T^N\times\bS^1)$, which disproves $(CE)_p$.  
Note that this reasoning breaks down for even $m$ since, for any given $M\in\N$, the inequality $1\leq
  ||k|^{2m}-l^2|\leq M$ holds for at most finitely many tuples $(k,l)\in\Z^N\times\Z$. The case of odd
  $m\geq 3$ remains open. }\end{rem}

\medskip
   
Next we consider the case where $M = \mathbb{S}^N$ for some $N \ge 2$ and $\cA=(-\Delta)^m$. 
Since the set of eigenvalues of $-\Delta$ on $\mathbb{S}^N$ is given by $\{k(k+N-1)\::\: k \in \N_0\}$ and is
 therefore integer-valued, the set of eigenvalues of $(-\Delta)^m$ on $\mathbb{S}^N$ contains an infinite
number of squares if $m \in \N$ is even, and this implies again that $\dim E^0 = \infty$ in this case. We have the
following result.

\begin{thm}
  \label{thm-intro-sphere}
  Let $M= \mathbb{S}^N$, $N \ge 2$, let $\cA=(-\Delta)^m$ for some even $m \in \N$, and let
   $2 < p < \frac{2(N+1)}{(N-m)_+}$.  Moreover, suppose that $q \in
   L^\infty(\mathbb{S}^N \times \mathbb{S}^1)$ is nonnegative and satisfies
    \begin{equation}
 \label{eq:q-con-SN-intro}
\inf \limits_{\omega \times \mathbb{S}^1} q  >0 \qquad \text{for some open subset $\omega \subset \mathbb{S}^N$ which meets every great
circle of $\mathbb{S}^N$}.
 \end{equation}
  Then (\ref{eq:nonlinear-wave-intro}) admits a ground state solution on $\mathbb{S}^N \times \mathbb{S}^1$.
\end{thm}

We assume $N \ge 2$ here since the case $N=1$ is already covered by Theorem~\ref{thm-S-1-intro}, which only
requires the weaker condition~(\ref{eq:q-con-S1}) in place of (\ref{eq:q-con-SN-intro}). We remark that
(\ref{eq:q-con-SN-intro}) is for instance satisfied if $\omega$ is an open neighborhood of a great
circle in $\mathbb{S}^N$.
Theorem~\ref{thm-intro-sphere} excludes the case $m=1$ which leads to the classical wave operator
  $-\Delta + \partial_{tt}$. In fact, in this case  we have $\dim E^0 < \infty$ as the set
 $\{k(k+N-1)\::\: k \in \N_0\}$  of eigenvalues of $-\Delta$ contains at most finitely many
 squares.  The latter follows from the fact that, for large $k$, a solutions $l$ of  $k(k+N-1)=l^2$ must satisfy $k+\frac{N-1}{2}-\frac{1}{2}< l< k+\frac{N-1}{2}$ and therefore cannot be integer.
 
 \medskip

 If $N \in \N,N\geq 3$ is odd and
\begin{equation}
  \label{eq:def-c-N}
c_N:=  \Bigl(\frac{N-1}{2}\Bigr)^2,
\end{equation}
then $-c_N$ is a non-zero eigenvalue of infinite multiplicity of $-\Delta +\partial_{tt}$ on $\mathbb{S}^N
\times I$ since 
$$
k(k+N-1)+ c_N  = (k+\frac{N-1}{2})^2
$$
is a square for every $k \in \N$. Hence condition $(CE)_p$ cannot be satisfied for  $-\Delta +\partial_{tt}$
in this case. 
However, we may consider the Schr\"odinger operator $\cA=-\Delta + c_N$ in (\ref{eq:nonlinear-wave-intro}),
which then leads to a nonlinear Klein-Gordon equation with mass $c_N$. 

\begin{thm}
  \label{thm-intro-sphere-klein-gordon}
  Let $M= \mathbb{S}^N$ for some odd $N \ge 3$, let $\cA=-\Delta+c_N$ with $c_N$ given in (\ref{eq:def-c-N}), and let
  $2 < p <  \frac{2(N+1)}{N-1}$.  Moreover, suppose that $q \in L^\infty(\mathbb{S}^N
  \times \mathbb{S}^1)$ is nonnegative and satisfies
    \begin{equation}
 \label{eq:q-con-SN-intro-1}
\inf \limits_{\omega \times \mathbb{S}^1} q  >0 \qquad \text{for some open subset $\omega \subset
\mathbb{S}^N$ that meets every great circle of $\mathbb{S}^N$}.
 \end{equation}
  Then (\ref{eq:nonlinear-wave-intro}) admits a ground state solution on $\mathbb{S}^N \times \mathbb{S}^1$.
\end{thm}

We note that the upper bound   $\frac{2(N+1)}{N-1}$ for $p$ is needed for the
verification of the compactness condition $(CE)_p$. We summarize our applications of
Theorem~\ref{thm-general-hyperbolic-intro} in a table, where $p^*$ denotes the exponent such that we can
prove the existence of ground states whenever $2<p<p^*$.  

\medskip 

  \begin{center}
  \captionof{table}{    
     Applications of Theorem~\ref{thm-general-hyperbolic-intro} \label{tab:results} }
     \renewcommand{\arraystretch}{1.5}
     \begin{tabular}{|c|c|c|c|c|c|c|c|}
     \hline
    ~ & $M$ & $\cA$ & $p^*$ & $q$ & $(CE)_p$  & $(CC)_q$    \\  \hline
     Thm~\ref{thm-wave-1+1} & $\mathbb{S}^1$ & $-\Delta$ & $\infty$ & \eqref{eq:q-con-wave-1+1} &
     Cor.~\ref{cor-compact-embedding-m-even-S-N} &  Thm~\ref{thm:CCq_for_wave1+1}
     \\     
    Thm~\ref{thm-S-1-intro} & $\mathbb{S}^1$ & $(-\Delta)^m$, $m\geq 2$ & $\infty$ & \eqref{eq:q-con-S1} & 
     Cor.~\ref{cor-compact-embedding-m-even-S-N}  & Cor.~\ref{cor-summary-CC-q}(iii),(iv)   \\
     Thm~\ref{thm-T-N-intro} & $\T^N$ &  $(-\Delta)^m$, $m$ even &
     $\frac{2N}{(N-m)_+}$ & \eqref{eq:q-con-TN} & Cor.~\ref{cor-compact-embedding-torus} &
     Cor.~\ref{cor-summary-CC-q}(ii),(iv)\\
     Thm~\ref{thm-intro-sphere} & $\mathbb{S}^N$, $N \ge 2$ & $(-\Delta)^m$, $m$ even& $
     \frac{2(N+1)}{(N-m)_+}$ & \eqref{eq:q-con-SN-intro} &
     Cor.~\ref{cor-compact-embedding-m-even-S-N} & Cor.~\ref{cor-summary-CC-q}(i)   \\
     Thm~\ref{thm-intro-sphere-klein-gordon} & $\mathbb{S}^N$, $N \ge 3$ odd & $-\Delta+c_N$ &
      $\frac{2(N+1)}{N-1}$    &       
     \eqref{eq:q-con-SN-intro-1} & Cor.~\ref{cor-compact-embedding-N-odd-S-N} & Cor~\ref{cor-summary-CC-q}(i)
     \\    \hline      
  \end{tabular} 
  \end{center}

\medskip

\begin{rem}
\label{comparison-remark}
  {\rm 
(i) In all of our results stated above, the pure power nonlinearity $q(x,t)|u|^{p-2}u$ in (\ref{eq:nonlinear-wave-intro}) can be replaced by the more general variant $q(x,t)f(u)$ with $f\in C(\R)$ satisfying the following conditions. 
   \begin{itemize}
   \item[$(f_1)$] The function $\frac{f(r)}{|r|}$ is strictly increasing on $\R$ with $\lim \limits_{r \to 0} \frac{f(r)}{|r|} = 0$.
   \item[$(f_2)$] There exists constants $c,C>0$ with 
        $$
         \int_0^s f(\tau)d\tau \geq c |s|^{p}\quad \text{and}\quad |f(s)|\leq C(1+|s|^{p-1}) \qquad \text{for all $s \in \R$.}
        $$
   \end{itemize}
   The more general version of Theorem~\ref{thm-general-hyperbolic-intro} with these assumptions is given in Theorem~\ref{thm-general-hyperbolic-section} below. The conditions $(f_1)$, $(f_2)$ include quasipolynomials of the form
   $$
   f(u)= \sum_{i=1}^k a_i |u|^{p_i-2}u\qquad \text{with}\qquad \text{$2 < p_1 <p_2 < \dots < p_k=
   p\quad$ and $\quad a_1,\dots,a_k > 0$.} $$
 (ii) As mentioned above, the abstract theory in \cite{bartsch-mederski-2017-JFPT} cannot be applied to
 obtain the results above. This is due to the key assumption (I8) in \cite[Section
 4]{bartsch-mederski-2017-JFPT} which is not satisfied in the present setting. In fact, in the context of
 (\ref{eq:nonlinear-wave-intro}), this assumption reads as 
 $$
 \frac{I(t_n u_n)}{t_n^2} \to \infty \qquad \text{if $t_n \to +\infty$ and $u_n^+ \to u^+ \not = 0$,}
 $$
 where the functional $I \in C^1(E,\R)$ is defined by $I(u)=\int_{M \times \bS^1}q(x)|u|^p\,d(x,t)$. So this
 assumption can only be satisfied if $I(u)>0$ for every $u \in E^+$. This, however, is wrong if $q \equiv 0$ on $U \times \bS^1$ for an open set $U \subset M$.
 To see this, choose $k_0 \in \N_0$ with $\nu_k(\cA) >0$ for $k > k_0$. Then take a function $\tilde u\in C^\infty_c(U)$ satisfying
 $$
 \int_{M} \tilde u \zeta_k \,dx = 0 \qquad \text{for $k=0,...,k_0$.}
 $$
 This is possible since $C^\infty_c(U)$ is infinite-dimensional. Exploiting the eigenfunction expansion of $\tilde u$,
 we find that $u \in E^+$ for the function $(x,t) \to u(x,t):=\tilde u(x)$, and $I(u)=0$ since $u$ is supported in $U \times \bS^1$.}   
\end{rem}

To close this introduction, we wish to compare the present work with the previously available critical point theory for strongly indefinite functionals and its application to the existence of time-periodic solutions to superlinear wave equations.
The direct variational approach goes back to the seminal papers by Rabinowitz \cite{Rab_FreeVibrations} and Benci-Rabinowitz \cite{benci-rabinowitz-1979}. It has then been developed further by many authors, including Hofer \cite{hofer:1983}, Li-Willem \cite{Li-Willem-1995}, Li-Szulkin \cite{Li-Szulkin-1996},   Bartsch-Ding \cite{BaDi_PeriodicSolutions} and Bartsch-Ding-Lee \cite{BaDiLee_PeriodicSolutions}. Moreover, the dual variational approach to superlinear wave equations has been introduced by Brezis-Coron-Nirenberg \cite{BreCorNir}.

In most of the works on periodic solutions to superlinear wave equations, the Dirichlet problem in bounded subsets of $\R^N$ is considered. An exception is the work of Zhou \cite{Zhou_waveSn} where the author considered a variant of (\ref{eq:nonlinear-wave-intro}) with the operator $\cA=-\Delta+c_N$ on a higher-dimensional sphere $M=\mathbb{S}^N$ with a direct variational approach, and he proves the existence of time-periodic solutions for a class of superlinear nonlinearities. Our Theorem~\ref{thm-intro-sphere-klein-gordon} complements his results in \cite{Zhou_waveSn} by providing the existence of ground states and admitting coefficient functions $q(x,t)$ that may vanish on open subsets of $\bS^N\times\bS^1$.
As we mentioned already, we are not aware of any previously available abstract result which applies to the double degenerate case where the linear wave operator in (\ref{eq:nonlinear-wave-intro}) has an infinite-dimensional kernel and the weight $q$ has a nontrivial zero set.

\subsection{Organization of the paper}

The paper is organized as follows. In Section~\ref{sec:an-abstr-exist} we formulate an abstract result on the
existence and variational characterization of ground state solutions of functionals of the form
(\ref{eq:general-functional-intro}). In Section~\ref{sec:exist-ground-state}, we apply this result in the
framework of generalized nonlinear wave equations, and we prove a generalization of
Theorem~\ref{thm-general-hyperbolic-intro}, see Theorem~\ref{thm-general-hyperbolic-section} below. In
Section~\ref{sec:compact-embeddings} we provide a general criterion for compact embeddings of the spaces
$E^\pm$ defined above into weighted $L^p$-spaces, and we apply this criterion to the specific examples
addressed in Theorems~\ref{thm-S-1-intro},~\ref{thm-T-N-intro},~\ref{thm-intro-sphere} and \ref{thm-intro-sphere-klein-gordon} above. 
Section~\ref{sec:control-property} is devoted to the verification of the control condition $(CC)_q$  in
these examples.
Finally, in Section~\ref{sec:proofs-theor-refthm} we complete the proofs of Theorems~\ref{thm-S-1-intro},~\ref{thm-T-N-intro},~\ref{thm-intro-sphere} and \ref{thm-intro-sphere-klein-gordon}.

\section{An abstract existence result for ground state solutions}
\label{sec:an-abstr-exist}
In this section we establish an existence theorem for ground state solutions in an abstract setting which allows
us to derive the main results presented in the introduction.

Let $(E,\|\cdot\|)$ be a reflexive Banach space with the following further properties.
\begin{itemize}
\item[(E1)] We have $\|u\|^2 = \|P^+ u\|^2+\|P^-u\|^2 + \|P^0 u\|^2$ for $u \in E$, where $P^j: E \to E^j
\subset E$, $j \in \{\pm,0\}$ are continuous projections associated to a topological decomposition $E = E^+
\oplus E^0 \oplus E^-$ into Banach spaces. Furthermore, the map $\|\cdot\|^2$ is continuously
differentiable on $E^+\oplus E^-$.  
\item[(E2)] The closed subspace $(E^-,\|\cdot\|)$ is uniformly convex.
\end{itemize}
We recall that $(E2)$ implies that for every sequence $(u_n)_n$ in $E^-$ with $u_n \weak u$ and $\|u_n\| \to
\|u\|$ we have $u_n \to u$ strongly as $n \to \infty$.

\medskip

In the following, to abbreviate the notation, we write $u^\pm$ in place of $P^\pm u$ for $u \in E$ and $u^0$ in place of $P^0 u$. We consider functionals $\Phi \in C^1(E)=C^1(E,\R)$ of the form
\begin{equation}\label{eq:defPhi}
u \mapsto \Phi(u)=\frac{1}{2}\|u^+\|^2-\frac{1}{2}\|u^-\|^2-I(u),
\end{equation}
where  $I$ satisfies the following assumptions: 
\begin{itemize}
\item[(I0)] {\em (Growth conditions)} $I \in C^1(E)$ is nonnegative with $I(0)=I'(0)=0$ and
  \begin{equation}
    \label{eq:B02-prelim}
I(u)>0 \qquad \text{for }u \in E^0 \setminus \{0\}.
\end{equation}
Moreover, for each $u
\in E \setminus \{0\}$, the function $t \mapsto \frac{I(tu)}{t^2}$ is nondecreasing on $(0,\infty)$ with
  \begin{equation}
    \label{eq:B01}
\limsup_{t \to 0} \frac{I(tu)}{t^2}<\frac{\|u\|^{2}}{2}
\qquad  \text{and}\qquad  
  \lim_{t \to +\infty} \frac{I(tu)}{t^2} = +\infty \quad \text{if $I(u) \not = 0$ and $u^+ \neq 0$.}
\end{equation}
\item[(I1)] {\em (Weak subadditivity property)} For every $\eps>0$ there exists $\kappa_\eps>0$ with $I(\phi+ \psi)\le \eps + \kappa_\eps(I(\phi)+ I(\psi))$ for all  $\phi, \psi \in E$.
\item[(I2)] {\em (Partial strong continuity)} $I$ is strongly continuous w.r.t. weak convergence in $E^+$, i.e. for every $u \in E$ and every sequence $(w_n)_n$ in $E^+$ with $w_n \weak w$ in $E^+$ we have $I(u + w_n) \to I(u+w)$ as $n \to \infty$.
\item[(I3)] {\em (Strict weak lower semicontinuity property w.r.t. $E^0$)} If $(u_n)_n$ is a sequence in
  $E$ with $u_n \weak u$ and $u_n^+ \to u^+$, then
  \begin{equation}
    \label{weak-lower-semicont-eq}
  I(u) \le \liminf_{n \to \infty} I(u_n).
  \end{equation}
  Moreover, if, in addition, $u_n^- \to u^-$ and equality holds in (\ref{weak-lower-semicont-eq}),
then $u_n^0 \to u^0$ strongly as $n \to \infty$ after passing to a subsequence.
\end{itemize}
We also need the following abstract geometric assumptions on the functional $\Phi$ defined in (\ref{eq:defPhi}).
\begin{itemize}
\item[(S1)] {\em (Saddle point structure I)} For each $w\in E^+ \setminus \{0\}$ the restriction of $\Phi$ to 
  \begin{equation}
    \label{eq:def-E-w}
E_w:= \R^+ w\plus
  E^0\plus E^-
  \end{equation}
has at most one critical point\footnote{Here and in the following  $\R^+ = (0,\infty)$ denotes the {\em open}
half line.}, i.e., there exists at most one $u \in E_w$ with $\Phi'(u) \equiv 0$ on $\R w\oplus E^0\oplus E^-$. Moreover, if $u$ exists, it is a global maximizer of $\Phi$ on $E_w$.
\item[(S2)] {\em (Saddle point structure II)} We have
  \begin{equation*}
  \sup_{E_w} \Phi<\infty \qquad \text{for at least one $w\in E^+ \setminus \{0\}$.}
  \end{equation*}
\end{itemize}
 
The following is the main result of this section. 
\begin{thm}
\label{main-abstract-result}
Suppose that (E1),(E2),(I0)--(I3) and (S1),(S2) are satisfied. Then we have the following:
\begin{itemize}
\item[(i)] The associated Nehari-Pankov set
$$  
\cNP:= \Bigl \{u \in E \setminus (E^{-} \oplus E^0) \::\: \Phi'(u)\Big|_{E^- \oplus E^0 \oplus \R u}= 0
\Bigr\}, $$
is non-empty and contains all non-zero critical points of $\Phi$.    
\item[(ii)] We have  
$$
c:= \inf_{\cNP} \Phi = \inf_{w \in E^+ \setminus \{0\}} \sup_{E_w} \Phi \:>\:0.
$$
Moreover, this value is attained, and every minimizer of $\Phi$ on $\cNP$ is a critical point of $\Phi$. In
particular, there is a ground state solution of the equation $\Phi'(u)=0$,
i.e., a minimizer within the set of non-zero critical points of $\Phi$.
\end{itemize}
\end{thm}

The remainder of this section is devoted to the proof of Theorem~\ref{main-abstract-result}. As anticipated in the introduction, we shall make use of a nonlinear saddle point reduction for which the assumptions (S1) and (S2) are essential.
The strategy is as follows. We first define the sets
\begin{equation} \label{eq:defcMplus}
  \cM^+:= \{w\in E^+  \setminus \{0\}: \sup_{E_w} \Phi < \infty\}
\end{equation}
and 
\begin{equation} \label{eq:defMplus}
  M^+:=  \cM^+ \cap S^+,
\end{equation}
where $S^+$ denotes the unit sphere in $E^+$. We then show that for any $w\in \cM^+$ -- and in particular
for any $w\in M^+$ -- the supremum appearing in~\eqref{eq:defcMplus} is positive and it is attained at some
uniquely determined maximizer $\hat m(w)$, which is a nontrivial critical point of the restriction of $\Phi$ to $E_w$ and therefore contained in $\cNP$. Afterwards, we minimize the reduced functional $w \mapsto \Psi(w):=\Phi(\hat m(w))$  on $M^+$ and prove that a minimizer exists. The final and most difficult task in the analysis is to prove that any such minimizer $w$
gives rise to a ground state solution $u = \hat m(w)$ of the equation $\Phi'(u)=0$. This will be done by
showing the continuity of the map $w \mapsto \hat m(w)$ and the Fr\'{e}chet differentiability of the reduced
functional $\Psi$.\\[0.1cm] 

\subsection{Preliminary observations}

In this subsection, we always assume that (E1),(E2) and (I0)--(I3) hold whereas the assumptions (S1),(S2) will
only be needed later. 
We start with the following simple remarks. 

\begin{rem}  \label{rem-B-1}{\rm 
  (i) If $I(u)=0$ for some $u \in E$, then (I1)  gives, inductively, $I (n u)= 0$ for all $n \in \N$. Moreover, the
  monotonicity of $t \mapsto \frac{I(tu)}{t^2}$ assumed in (I0) then implies that $I(tu)=0$ for all $t>0$.\\
  (ii) Assumptions (I0) and (I3) imply that 
  \begin{equation}
    \label{eq:B02}
  \liminf_{u \in E^0, \|u\| \to \infty} \frac{I(u)}{\|u\|^2} >0. 
\end{equation}
Indeed, suppose by contradiction that
\begin{equation}
  \label{eq:contradiction-statement}
\frac{I(u_n)}{\|u_n\|^2} \to 0 
\end{equation}
for a sequence $(u_n)_n$ in $E^0$ with $t_n:= \|u_n\| \ge 1$ for all $n \in \N$ and $t_n \to +\infty$ as $n \to \infty$. Then $v_n:= \frac{1}{t_n}u_n$ satisfies $\|v_n\|=1$ for all $n \in \N$, and we may pass to a subsequence with $v_n \weak v$ in $E^0$. We claim that
\begin{equation}
  \label{eq:B02-help}
I(v_n) \ge c \qquad \text{for all $n \in \N$ with a constant $c>0$,}
\end{equation}
since otherwise $\liminf \limits_{n \to \infty}I(v_n)= 0$, and therefore $I(v)=0$ and $v_n \to v$ strongly by (I3), whereas $v=0$ by (I0). This contradicts the fact that $\|v_n\|=1$ for all $n \in \N$. Hence (\ref{eq:B02-help}) is true, and now 
the monotonicity condition in (I0) implies, since $t_n \ge 1$, that
$$
\frac{I(u_n)}{\|u_n\|^2} = \frac{I(t_n v_n)}{t_n^2} \ge I(v_n) \ge c \qquad \text{for all $n \in \N$,}
$$
in contradiction to (\ref{eq:contradiction-statement}). Hence (\ref{eq:B02}) is true.}
\end{rem}

The next preliminary observation will be used at various places in the remainder of this section. 

\begin{prop} \label{prop:IPositivity}
Let $v\in E$ satisfy $v^+ \in \cM^+$ and $\|v^-\|\leq \|v^+\|$. Then $I(v)>0$.
\end{prop}

\begin{proof}
Let $w:= v^+ \in M^+$. We argue by contradiction and assume $I(v)=0$.   Then $\|v^-\| < \|v^+\|$
 cannot be true because for $t>0$ we have $tv \in E_w$ and $I(tv)=0$ by Remark~\ref{rem-B-1}. Consequently, 
  $$ 
    \sup_{E_{w}} \Phi
    \geq \Phi(t v) = t^2 \bigl(\|v^+\|^2 - \|v^-\|^2\bigr) \to +\infty
    \quad \text{as }t \to \infty, 
  $$
  which contradicts the fact $w\in \cM^+$, see~\eqref{eq:defcMplus}. The case $\|v^-\|=\|v^+\|$  cannot occur
  either because, for $n \in \N$, we have  $v_n :=  n v + v^+ = (n+1)v^+-v^-\in E_w$
  and $I(n v + v^+) \le 1 + \kappa_1 \bigl(I(nv)+I(v^+)\bigr) = 1+ \kappa_1 I(v^+)$ by (I1) and
  Remark~\ref{rem-B-1}, which implies that
\begin{align*}
   \sup_{E_w} \Phi 
  &\geq \Phi(v_n) 
  =  \frac{1}{2}(\|(n+1)v^+\|^2 - \|n v^-\|^2) 
  - I(nv + v^+)\\
  &\ge   \frac{1}{2} \left((n+1)^2-n^{2}\right) \|v^+\|^2- \kappa_1 I(v^+)-1 \\
  &\ge  n \|v^+\|^2 - \kappa_1
  I(v^+)-1   \to \infty \qquad \text{as }n \to + \infty.
\end{align*}
So this contradicts again  $w \in \cM^+$. Hence we have $I(v)>0$.
\end{proof}

Next we prove a key criterion for the boundedness of sequences in $E$.

\begin{lem}  \label{lem:boundedness}
  Let $(u_n)_n$ be a sequence in $E \setminus \{0\}$ with $\Phi(u_n) \ge 0$ for all $n \in \N$.
  \begin{itemize}
  \item[(i)] If $\frac{u_n^+}{\|u_n\|}\to 0$ as $n \to \infty$, then $(u_n)_n$ is bounded.
  \item[(ii)] If $\frac{u_n^+}{\|u_n\|} \to w$ as $n \to \infty$ for some $w \in \cM^+$, then $(u_n)_n$ is bounded.  
  \end{itemize}
\end{lem}

\begin{proof}
  We first note that the assumption $\Phi(u_n)\ge 0$ implies, thanks to $I\geq 0$ by (I0), that
  $\|u_n^+\| \ge   \|u_n^-\|$ for all $n \in \N$. 
  Setting $v_n = \frac{u_n}{\|u_n\|}$, we then also have 
  \begin{equation}
    \label{eq:boundedness-eq-2}
  \|v_n^+\| \ge   \|v_n^-\|\qquad \text{for all $n \in \N$.}
  \end{equation}
  To prove (i) we first suppose that $v_n^+\to 0$ holds. Then \eqref{eq:boundedness-eq-2}
  gives $v_n^- \to 0$ as $n \to \infty$, which implies that $\|v_n^0\| \to 1$
  as $n \to \infty$. Without   loss of generality, we may assume that $\inf \limits_{n \in \N} \|v_n^0\|>0$.
   Then, by \eqref{eq:B02}, there exists $\eps_0, t_0>0$ with 
\begin{equation}
  \label{eq:first-remark-eq-2}
    \frac{I(t v_n^0)}{t^2} \ge \eps_0 \qquad \text{for all $n \in \N$, $t \ge t_0$.}
\end{equation}
Applying $(I0)$ and $(I1)$ with $\eps = \frac{t_0^2 \eps_0}{2}$ and $\kappa := \kappa_{\eps}$, we find for
every fixed $t>0$
$$
I(t v_n^0) = I\bigl(tv_n  - t(v_n^+ + v_n^-)\bigr) \le \kappa \Bigl(I(t v_n) + I(- t (v_n^+
+ v_n^-))\Bigr)+  \frac{t_0^2 \eps_0}{2} \le \kappa I(t v_n) + \frac{t_0^2 \eps_0}{2}+ o(1) $$
as $n \to \infty$. Hence, for all $t>0$
\begin{equation}
  \label{eq:first-remark-eq-1}
    I(t v_n) \ge \frac{1}{\kappa}\Bigl(I(t v_n^0)-\frac{t_0^2 \eps_0}{2}\Bigr) -o(1) \qquad \text{as $n \to \infty$.}
\end{equation}
Combining (\ref{eq:first-remark-eq-1}) and (\ref{eq:first-remark-eq-2}) yields, by choosing $t=t_0$, 
    $$
    \frac{I(t_0 v_n)}{t_0^2} \ge \frac{1}{\kappa}\Bigl(\frac{I(t_0 v_n^0)}{t_0^2}-\frac{\eps_0}{2}\Bigr)
    -o(1)\ge  \frac{\eps_0}{2 \kappa}-o(1) \qquad \text{as }n \to \infty.
    $$
Now assume for contradiction that, after passing to some
subsequence, $s_n:= \|u_n\| \to \infty$. The monotonicity assumption in $(I0)$ then implies
    $$
    \frac{I(s_n v_n)}{s_n^2} \ge 
    \frac{I(t_0 v_n)}{t_0^2}  \geq \frac{\eps_0}{2 \kappa}-o(1) \qquad \text{as }n \to \infty
    $$
    and therefore
    $$
    \Phi(u_n)= \Phi(s_n v_n) = s_n^2 \Bigl(\frac{\|v_n^+\|^2 - \|v_n^-\|^2}{2} - \frac{I(s_n v_n)}{s_n^2} \Bigr) \le
    s_n^2 \Bigl(o(1)- \frac{\eps_0}{2\kappa}\Bigr)  \to -\infty
    $$
    as $n \to \infty$. This contradicts our assumption $\Phi(u_n)\geq 0$, so $(u_n)_n$ must be bounded and (i)
    is proved. 

    To prove (ii), we assume $v_n^+\to w\in\mathcal M^+$. Since $(v_n)$ is normalized, we may pass to a
    subsequence such that $v_n \weak v$ in $E$, which by assumption (ii) yields 
    $$
    v_n^+ \to v^+ = w, \quad v_n^- \weak v^- \quad \text{and}\quad v_n^0 \weak v^0 \qquad
    \text{as } n \to \infty.
    $$
    It then follows from (\ref{eq:boundedness-eq-2}) that  
  $$
    \|v^-\|\leq \liminf_{n\to\infty} \|v_n^-\|
    \leq \liminf_{n\to\infty} \|v_n^+\|
    = \|v^+\|,
  $$ 
  and Proposition~\ref{prop:IPositivity} gives $I(v)>0$. Moreover, by the weak lower semicontinuity of $I$,
  \begin{equation}
    \label{eq1-lem:boundedness}
  \liminf_{n\to\infty} \frac{I(t v_n)}{t^2} \geq \frac{I(tv)}{t^2} \qquad \text{for every } t>0.
  \end{equation}
  Then $(I0)$ and $I(v)>0$ imply 
  \begin{equation}
    \label{eq2-lem:boundedness}
    \frac{I(tv)}{t^2} \to \infty \qquad \text{as }t \to +\infty.
  \end{equation}
  Supposing again, by contradiction, that $s_n:= \|u_n\| \to \infty$ after passing to a subsequence, we then
  deduce from (\ref{eq1-lem:boundedness}) and the monotonicity condition in {$(I_0)$} that $$
    \frac{I(s_n v_n)}{s_n^2} \ge \frac{I(t v_n)}{t^2} \ge \frac{I(tv)}{t^2} +o(1) \qquad \text{as }
    n \to \infty \text{ for every }t>0 
  $$
  and hence, by (\ref{eq2-lem:boundedness}), 
  $$
    \frac{I(s_n v_n)}{s_n^2} \to \infty \qquad \text{as }n \to \infty.
  $$
  As a consequence,
  \begin{align*}
\limsup_{n \to \infty}\Phi(u_n) = \limsup_{n \to \infty} \Phi(s_n v_n)
&= \frac{s_n^2}{2}\limsup_{n \to \infty}\Bigl((\|v_n^+\|^2-\|v_n^-\|^2) -\frac{I(s_n v_n)}{s_n^2}\Bigr)\\
& \leq \frac{s_n^2}{2}\limsup_{n \to \infty}\Bigl(\|v^+\|^2  -\frac{I(s_n v_n)}{s_n^2}\Bigr) = -\infty,
  \end{align*}
  contrary to our assumption  $\Phi(u_n)\geq 0$. So $(u_n)_n$ is bounded and (ii) is proved.
\end{proof}

Next we show that for $w \in \cM^+$ the supremum of $\Phi$ over $E_w$ is attained.

\begin{lem} \label{lem:existence-maximizer}
  Let $w \in \cM^+$, and let $E_w$ be defined as in~(\ref{eq:def-E-w}).
   
   Then  $\Phi|_{E_w}$ admits a global maximizer $u \in E_w$
satisfying
  \begin{equation}
    \label{eq:lem-existence-maximizer}
    \Phi'(u)[tw+y^0+y^-] = 0 \qquad \text{for all $t\in\R,\;y^0\in E^0,\;y^-\in E^-.$}
  \end{equation}
\end{lem}

\begin{proof} Without loss of generality, we may assume that $\|w\|=1$, so $w \in M^+$. Let $u_n:= t_n w +
u_n^0 + u_n^-$ be a maximizing sequence for $\Phi|_{E_w}$.
Since 
 \begin{equation}\label{eq:Phinear0}
   \sup_{E_w} \Phi
   \geq \Phi(\tau w)=\frac{\tau^2}{2}-I(\tau w)>0 \qquad
   \text{for some small enough } \tau= \tau(w)>0
 \end{equation}
 by assumption (I0), it follows that $\Phi(u_n)$ is positive for almost all $n\in\N$. Moreover, since $\frac{u_n^+}{\|u_n\|} = \frac{t_n}{\|u_n\|} w$ and the sequence $\bigl(\frac{t_n}{\|u_n\|}\bigr)_n$ is bounded, 
 we may pass to a subsequence such that one of the assumptions (i) and (ii) in Lemma~\ref{lem:boundedness} is
 satisfied. Hence it follows that $(u_n)_n$ is bounded, and we may pass to a subsequence with $u_n \weak u$.
 The weak limit is clearly of the form $u=t w + u^0+u^-\in \overline{E_w}$, where $t_n \to t\geq 0$.  Hence, the weak lower semicontinuity
 of the norm and of $I$, as ensured by assumption $(I3)$, imply $$
  \Phi(u)
  = \frac{1}{2}\Bigl(t^2 - \|u^-\|^2\Bigr)- I(u) 
  \ge \lim_{n \to \infty} \frac{1}{2}\Bigl(t_n^2 -  \|u_n^-\|^2\Bigr)- I(u_n) 
  = \lim_{n \to \infty} \Phi(u_n)
  = \sup_{E_w} \Phi, 
  $$
where the right hand side is positive by (\ref{eq:Phinear0}). A posteriori, it follows that $t>0$, since
otherwise $\Phi(u) \le 0$. Hence $u \in E_w$ is a maximizer of $\Phi|_{E_w}$. In particular, it is a critical point of $\Phi|_{E_w}$ and therefore satisfies (\ref{eq:lem-existence-maximizer}). 
\end{proof}

\subsection{The saddle point reduction and ground states}
\label{sec:saddle-point-reduct}

{In this subsection, we always assume that (E1),(E2), (I0)--(I3) and (S1), (S2) hold.}

We first note that the sets $\cM^+$ and $M^+$ defined in (\ref{eq:defcMplus}) and (\ref{eq:defMplus}) are
nonempty by assumption (S2). Combining {Lemma~\ref{lem:existence-maximizer}} with assumption (S1), we also
find that for $w \in \cM^+$ the functional $\Phi|_{E_w}$ admits a {\em unique} global maximizer $\hat m(w) \in E_w$.
As outlined in the beginning of this section, we now introduce the reduced functional
\begin{equation}
  \label{eq:def-Psi}
  \Psi: \cM^+ \to \R, \qquad \Psi(w) := \Phi(\hat m(w)),
\end{equation}
and we put
\begin{equation}
  \label{eq:def-c}
c:= \inf \limits_{\cM^+} \Psi = \inf \limits_{M^+} \Psi. 
\end{equation}
Here we used $\hat m(tw)=\hat m(w)$ for all $t\neq 0$.
Note that $c \ge 0$ as
\begin{equation}
  \label{eq:Phinear0-consequence}
  \Psi(w) = \Phi(\hat m(w))>0 \qquad \text{for all $w \in \cM^+$}
\end{equation}
by (\ref{eq:Phinear0}). The next proposition states that the reduced functional $\Psi$ admits a minimizer in $M^+$.

\begin{prop}\label{prop:conclusion}
The infimum $c$ in (\ref{eq:def-c}) is attained by $\Psi$ on $M^+$.
\end{prop}
\begin{proof}
  Let $(w_n)$ be a minimizing sequence for $\Psi$ on $M^+$, i.e.,
  $$
    \Psi(w_n) = \Phi(\hat m(w_n)) \to c  \qquad \text{as }n
    \to \infty.
  $$
  Since $\|w_n\|=1$ for all $n$, we may pass to a subsequence with $w_n \rightharpoonup w$ in $E^+$ and
  $\|w\| \le 1$.

  \textbf{Case 1:} $w = 0$. Then $t w_n \rightharpoonup 0$ for every $t>0$ and thus, exploiting the strong
  continuity assumption (I2) with $u=0$, we find that $I(t w_n)  \to 0$ for all $t>0$. Hence, $$
      c + o(1) 
      = \Phi(\hat m(w_n)) \ge  \Phi(t w_n) = \frac{t^2}{2} \|w_n\|^2 - I(t w_n)=
    \frac{t^2}{2} - o(1)
    \qquad \text{as }n\to \infty \text{ for every }t>0.
    $$
    This is a contradiction.
    
    \textbf{Case 2:} $w \not = 0$. For every $z^0 \in E^0$ and $z^- \in E^-$ and $t > 0$ we   have 
    \begin{align*}
    \Phi(t w + z^0+z^-) 
    &= \frac{1}{2}\Bigl(t^2\|w\|^2 - \|z^-\|^2\Bigr) -
                          I(t w + z^0+z^-)\\
      &\le \frac{1}{2}\Bigl(t^2 - \|z^-\|^2\Bigr) -
                                 I(t w + z^0+z^-)\\
      &= \lim_{n \to \infty}\Bigl[\frac{1}{2}\Bigl(\|t w_n\|^2 - \|z^-\|^2\Bigr) -
                                 I(t w_n + z^0+z^-)\Bigr]\\
                        &= \lim_{n \to \infty}    \Phi(t w_n + z^0+z^-) \le \lim_{n \to \infty} \Phi(\hat m(w_n))= c.
    \end{align*}
    Here again we used the strong continuity property from (I2). Given that the inequalities above must be equalities, we deduce
     $\|w\|=1$ as well as
    $$
      \sup_{E_w} \Phi \leq  c.
    $$
    Consequently, $w\in M^+$ and 
    $$
      \Psi(w) = \Phi(\hat m(w))  \leq  c = \inf_{M^+} \Psi.
    $$
    So $w$ is a minimizer of $\Psi$ on $M^+$.
 \end{proof}

 \begin{rem}\label{prop:positivity}{\rm 
A posteriori, it follows from Proposition~\ref{prop:conclusion} and (\ref{eq:Phinear0-consequence}) that 
$$
c = \inf_{M^+} \Psi > 0.
$$
Moreover, writing $\hat m(w)=s_w w + z^0+z^-$ for $w \in \cM^+$, we have
$$
\inf_{w \in M^+} s_w >0. 
$$
Indeed, since $I$ is nonnegative, the second claim follows
 from    
 $$
0< c\leq \Psi(w)
    = \frac{1}{2}(s_w^2-\|z^-\|^2)-I(\hat m(w))
    \leq \frac{1}{2}s_{w}^2 \qquad \text{for all }w \in M^+.
 $$}
\end{rem} 

Next, we show that the minimizer from Proposition~\ref{prop:conclusion} is indeed a ground state
  solution to our problem. In particular we  verify the validity of the Euler-Lagrange equation at this point.  
  
\begin{prop} \label{prop:Mplus_open}
  The set $\cM^+$ is open in $E^+$. Even more, for every $w \in \cM^+$ there exists a neighborhood
  $N$ of $w$ in $E^+$ with 
  $$
  N \subset \cM^+ \qquad \text{and}\qquad \sup_{\tilde w \in N}\Psi(\tilde w)<\infty.
  $$
\end{prop}
\begin{proof}
  Suppose by contradiction that there exists $w \in \cM^+$ and a sequence of points 
  $w_n \in E^+$ such that $w_n \to w$ and $\Psi(w_n)\to \infty$.
  Then there exist $t_n>0$, $u^0_n \in E^0,u_n^- \in E^-$ with $\Phi(u_n) \to \infty$ for
  $u_n:= t_n w_n + u^0_n + u^-_n$. Since $\frac{u_n^+}{\|u_n\|} = \frac{t_n}{\|u_n\|} w_n$ and the sequence $\bigl(\frac{t_n}{\|u_n\|}\bigr)_n$ is bounded, 
  we may pass to a subsequence such that $\frac{t_n}{\|u_n\|} \to t_* \ge 0$ and thus 
   $\frac{u_n^+}{\|u_n\|} \to t_* w$. Hence one of the assumptions (i) and (ii) of
   Lemma~\ref{lem:boundedness} is satisfied, and it thus follows that $(u_n)_n$ is bounded. But this is
   impossible since, by definition of $\Phi$ and the nonnegativity of $I$, the functional $\Phi$ is bounded
   from above on bounded subsets of $E$. The contradiction gives the claim.
\end{proof}
   
To prove differentiability properties of $\Psi$ on $M^+$ we first verify that $\hat m$ enjoys suitable
continuity properties.

\begin{prop}\label{prop:hatm} ~
  \begin{itemize}
  \item[(i)]   Let $(w_n)_n$ be a sequence in $\cM^+$ which satisfies $\sup \limits_{n \in \N} \Psi(w_n) <
  \infty$ and ${w_n}\to w \in E^+ \sm\{0\}$.
  Then $w \in \cM^+$, $\hat m(w_n)\to \hat m(w)$ and $\Psi(w_n) \to \Psi(w)$ as $n \to \infty$.
\item[(ii)] The map $\hat m$ is continuous on $\cM^+$.  
  \end{itemize}
\end{prop}
\begin{proof}
(i) Without loss of generality, by normalization and since $\hat m(w_n) = \hat m(\frac{w_n}{\|w_n\|})$, we may assume that $w_n \in M^+$ for every $n \in \N$.  For every $t>0, z^0 \in E^0$ and $z^- \in E^-$ we have, by the strong continuity of $I$ stated in assumption
  (I2),
    \begin{align*}
    \Phi(t w + z^0+z^-) &= \frac{1}{2}\Bigl(t^2\|w\|^2 - \|z^-\|^2\Bigr) -
                          I(t w + z^0+z^-)\\
      &= \lim_{n \to \infty}\Bigl[\frac{1}{2}\Bigl(\|t w_n\|^2 - \|z^-\|^2\Bigr) -
                                 I(t w_n + z^0+z^-)\Bigr]\\
                        &= \lim_{n \to \infty}    \Phi(t w_n + z^0+z^-)\le \liminf_{n \to \infty} \Phi(\hat m(w_n))= \liminf_{n \to \infty}\Psi(w_n).
\end{align*}
This yields $w\in M^+$ as well as $\Psi(w) \le \liminf
\limits_{n \to \infty}\Psi(w_n)<\infty$.  

It now remains to prove that
\begin{equation}
  \label{prop:hatm-eq-1}
\hat m(w_n)\to \hat m(w).  
\end{equation}
Once this is proved, it follows from the continuity of $\Phi$ that
$$
\Psi(w_n)= \Phi(\hat m(w_n)) \to \Phi(\hat m(w))=\Psi(w) \qquad \text{as }n \to \infty.
$$
By a standard argument, it suffices to prove (\ref{prop:hatm-eq-1}) for a subsequence.
As an intermediate step, we prove the weak convergence $\hat m(w_n)\wto \hat m(w)$
after passing to a subsequence. For this we put 
$$
    u_n := \hat m(w_n) = s_{w_n} w_n + u_n^0+u_n^- \quad \text{for }n \in \N\qquad \text{and}\qquad 
    u := \hat m(w) = s_w w + u^0+u^-.
  $$
Then $\Phi(u_n)=\Psi(w_n)>0$ for all $n$ by Remark~\ref{prop:positivity}. Moreover, since
$\frac{u_n^+}{\|u_n\|} = \frac{s_{w_n}}{\|u_n\|} w_n$ and the sequence $\bigl(\frac{s_{w_n}}{\|u_n\|} \bigr)_n$ is 
bounded as $w_n \to w \not = 0$, we may pass to a subsequence such that
$\frac{s_{w_n}}{\|u_n\|} \to s_* \ge 0$ and hence $\frac{u_n^+}{\|u_n\|} \to s_* w$. Hence one of the
assumptions (i) and (ii) of Lemma~\ref{lem:boundedness} is satisfied and it thus follows that $(u_n)_n$ is
bounded. After passing to a subsequence, we may therefore assume that $s_{w_n} \to s$ and $u_n\wto sw+y^0+y^-$ as $n \to \infty$ for some $s\in\R^+, y^0\in E^0, y^-\in E^-$. 
The convergences $s_{w_n}\to s$, $w_n\to w$, $u_n^- \weak {y^-}$, $u_n\wto sw+y^0+y^-$ and the weak lower
semicontinuity of $I$ then imply that
  \begin{align} 
    \Phi(u_n)&= \frac{s_{w_n}^2}{2}\|w_n\|^2- \frac{\| u_n^-\|^2}{2} - I(u_n) \leq \frac{s^2}{2}\|w\|^2- \frac{\|y^-\|^2}{2} - I(sw+y^0+y^-) + o(1)   \label{prop:hatm-eq-2}\\
             &= \Phi(sw+y^0+y^-)  +    o(1)             
             \leq \Phi(\hat m(w)) + o(1)
              \nonumber    \\
             &=\Phi(s_w w + u^0+u^-)+o(1)=\Phi(s_w w_n + u^0+u^-)+o(1) \nonumber\\
    &\le \sup_{E_{w_n}} \Phi +o(1)=  \Phi(u_n)+ o(1) \nonumber
  \end{align}
  as $n \to \infty$. Consequently, all    inequalities in this chain are (asymptotic) equalities, and this
  implies   
  $$
    \Phi(sw+y^0+y^-) =  \Phi(\hat m(w)) =  \Phi(u).
  $$
  Hence the uniqueness property {from assumption (S1)} 
   implies that $u= \hat m(w) = sw+y^0+y^-$, and therefore $u_n
  \wto u$, i.e., $\hat m(w_n)\wto \hat m(w)$. 
    \smallskip
  
  To prove strong convergence, we note that the equality in (\ref{prop:hatm-eq-2}) 
  and assumption $(E2)$ imply that
  \begin{equation}
  \label{prop:hatm-eq-3}  
  u_n^- \to u^- = y^- \quad \text{strongly in $E$}\qquad \text{and} \qquad I(u_n) \to I(u) = I(sw+y^0+y^-)
  \end{equation}
  as $n \to \infty$. So we know $u_n^+ = s_{w_n}w_n^+\to s_ww=u^+$, $u_n^-\to u^-$ and $I(u_n)\to I(u)$. Hence we deduce from assumption $(I3)$ that
  \begin{equation}
  \label{prop:hatm-eq-4}  
  u_n^0 \to u^0\qquad \text{strongly in $E$.}
\end{equation}
Since also $u_n^+ = s_{w_n} w_n \to s w = u^+$ strongly in $E$, we derive from (\ref{prop:hatm-eq-3}) and (\ref{prop:hatm-eq-4}) that
$$
\hat m(w_n)= u_n \to u= \hat m(w)\qquad \text{strongly in }E,
$$
as claimed in (i).\\
(ii) To see the continuity of $\hat m$ on $\cM^+$, let $w \in \cM^+$, and let $(w_n)_n$ be a sequence in
$\cM^+$ with $w_n \to w$. Then we have $(\Psi(w_n))_n$ is bounded by Proposition~\ref{prop:Mplus_open},
so (i) implies that $\hat m(w_n) \to \hat m(w)$ as $n \to \infty$.
\end{proof}

\begin{prop}\label{prop:Derivative}
We have $\Psi\in C^1(\cM^+)$ with $\Psi'(w)[h]=  s_w \Phi'(\hat m(w))[h]$ for all $w\in \cM^+, h\in E^+$, where $s_w$ is defined in Remark~\ref{prop:positivity}.   
\end{prop}
\begin{proof}
  Let $w \in \cM^+$ be fixed, and let $h \in E^+$.  By Proposition~\ref{prop:Mplus_open}, there exists $\eps=
  \eps(h)>0$ with $w+th\in \cM^+$ for $t\in \R$ with $|t|<\eps$. Moreover, we have $\hat m(w) + ts_w h \in
  E_{w+th}$ in this case since $\bigl(\hat m(w) + ts_w h\bigr)^+ = s_w (w + th)$. Thus for $t\in \R$ with $0< |t|<\eps$ we have 
  \begin{align*}
    \frac{1}{t}\left(\Psi(w + th)-\Psi(w)\right)
    &= \frac{1}{t}\left(\Phi(\hat m(w + th))-\Phi(\hat m(w))\right) \\
    &\geq \frac{1}{t}\left(\Phi(\hat m(w) + ts_w h)-\Phi(\hat m(w))\right) \\
    &=  \Phi'(\hat m(w)+\tau_t t s_wh)[s_w h] \qquad \text{with $\tau_t\in (0,1)$ chosen suitably.}
 \end{align*}
 Consequently,
 \begin{equation*}
 \liminf_{t \to 0}\frac{1}{t}\left(\Psi(w + th)-\Psi(w)\right) \ge s_w \Phi'(\hat m(w))[h].  
 \end{equation*}
 For the upper bound we note, similarly, that $\hat m(w+th) - t s_{w+th}h  \in E_w$ and therefore, for $0< |t|< \eps$,
  \begin{align*}
    &\frac{1}{t}\left(\Psi(w+th)-\Psi(w)\right) \\
    &\leq \frac{1}{t}\Big(\Phi(\hat m(w+th))-\Phi\big(\hat m(w+th) - t s_{w+th}h\big)\Big) \\
    &=   \Phi'(\hat m(w+th)-\tau_t t s_{w+th}h)[s_{w+th}h] 
\qquad \text{with $\tau_t\in (0,1)$ chosen suitably.}
  \end{align*}
  Since $\hat m(w+th)\to \hat m(w)$ and therefore $s_{w+th}\to s_w$ as $t \to 0$ by Proposition~\ref{prop:hatm}(ii), it follows that
 \begin{equation}
   \label{eq:gateaux-upper}
 \limsup_{t \to 0}\frac{1}{t}\left(\Psi(w + th)-\Psi(w)\right) \le s_w \Phi'(\hat m(w))[h].  
 \end{equation}
Here we used the fact that $\Phi': E \to E^*$ is continuous, which follows from $(E1)$ and $(I0)$.  Since the lower
bound coincides with the upper bound, one finds that the G\^{a}teaux-derivative of $\Psi$ at $w$ exists and
is given by 
$$
  \Psi'(w)[h] =   \lim_{t\to 0^+} \frac{1}{t}\left(\Psi(w+th)-\Psi(w)\right) =  s_w \Phi'(\hat m(w))[h].
$$
Since the map $\cM^+ \to \bigl(E^+\bigr)^*$, $w \mapsto s_w \Phi'(\hat m(w))\big|_{E^+}$ is continuous, it
then  follows that $\Psi \in C^1(\cM^+)$ with Fr\'{e}chet derivative given by the above formula. 
  This finishes the proof.
\end{proof}

\begin{prop}\label{prop:groundstate}
If $w \in \cM^+$ satisfies $\Psi(w)= c$ with $c$ given in (\ref{eq:def-c}), then $u_* := \hat m(w)$ is a ground state solution of the equation $\Phi'(u_*)=0$, i.e., it is a least energy nontrivial critical point of $\Phi$.
\end{prop}

\begin{proof}
  By Proposition~\ref{prop:Derivative} and the minimizing property of $w$, we have $\Psi'(w)[h]=s_w \Phi'(\hat m(w))[h]=0$ for all $h \in E^+$. 
  Since Remark~\ref{prop:positivity} implies $s_w>0$, we thus infer $\Phi'(\hat m(w))[h]=0$ for all $h\in
  E^+$. Combining this information with Lemma~\ref{lem:existence-maximizer} gives $\Phi'(\hat m(w))[h]=0$ for all $h\in E$, which shows
  that $u_*:= \hat m(w)$ is a critical point of $\Phi$.

  To see that $u_*$ is a least energy nontrivial critical point of $\Phi$, let $u$ be any other nontrivial critical point of $\Phi$. Then $u^+ \not =0$, since in case $u^+ =0$ we would have
  \begin{align*}
    0 &= \frac{d}{dt}\Big|_{t=1}\Phi(tu)=-\frac{d}{dt}\Big|_{t=1}\Bigl(\|(tu)^-\|^2+ I(tu)\Bigr)
    = -\frac{d}{dt}\Big|_{t=1}t^2 \Bigl(\|u^-\|^2+ \frac{I(tu)}{t^2}\Bigr) \\
    &= -2\Bigl(\|u^-\|^2+ I(u)\Bigr) 
      - \frac{d}{dt}\Big|_{t=1}\frac{I(tu)}{t^2}
    \le -2   \Bigl(\|u^-\|^2+ I(u) \Bigr)
  \end{align*}
as a consequence of $(I0)$, which would imply that $u^-=0$ and $I(u)=I(u^0)=0$.
It then follows from (I0) that $u=u^0=0$, a contradiction.   Hence $u^+ \not = 0$, and then $(S1)$
implies that $u^+ \in \cM^+$ and $u=\hat m(u^+)$. Hence, 
$$
  \Phi(u)= \Phi(\hat m(u^+))=\Psi(u^+)\geq c = \Psi(w)=\Phi(u_*)
 $$   
 shows that $u_*$ is a least energy nontrivial critical point of $\Phi$.
 \end{proof}

 \subsection{Completion of the proof of Theorem~\ref{main-abstract-result}}

 In this subsection we complete the proof of Theorem~\ref{main-abstract-result}, assuming (E1),(E2),(I0)--(I3) and (S1),(S2) in the following. To prove Part (i), we recall the definition 
$$  
\cNP:= \Bigl \{u \in E \setminus (E^{-} \oplus E^0) \::\: \Phi'(u)\Big|_{E^- \oplus E^0 \oplus \R u}= 0
\Bigr\}
$$
of the associated Nehari-Pankov set. As noted at the beginning of Subsection~\ref{sec:saddle-point-reduct}, it
follows from $(S2)$ that the set $M^+$ is nonempty, while Lemma~\ref{lem:existence-maximizer} and $(S1)$
imply that for every $w \in M^+$ the functional $\Phi|_{E_w}$ admits a unique global maximizer $u = \hat m(w)
\in E_w$ which satisfies 
$$
  \Phi'(u)\Big|_{E^- \oplus E^0 \oplus \R u}=0.
$$ 
This implies $\hat m(w) \in \cNP$ for every $w \in M^+$, so in particular the set $\cNP$ is nonempty.
Moreover, as in the proof of Proposition~\ref{prop:groundstate} we see that every nontrivial critical point $u$ of $\Phi$ satisfies $u^+ \not = 0$ and is therefore contained in $\cNP$.
This finishes the proof of Part~(i).

To prove Part (ii), we first note that the map $\hat m: M^+ \to \cNP$ is continuous by Proposition~\ref{prop:hatm}. In fact it is a homeomorphism with inverse given by the continuous map $\cNP \to M^+$, $u \mapsto \frac{u^+}{\|u^+\|}$. Moreover, by definition of the reduced functional $\Psi$ in (\ref{eq:def-Psi}) and Proposition~\ref{prop:conclusion}, it follows that the value
$$
c= \inf_{M^+}\Psi= \inf_{\cNP} \Phi = \inf_{w \in E^+ \setminus \{0\}} \sup_{E_w} \Phi 
$$
is attained on $M$, and every minimizer of $\Phi$ on $\cNP$ corresponds to a minimizer of $\Psi$ on $M^+$ and is therefore a least energy critical point of $\Phi$ by Proposition~\ref{prop:groundstate}. We also note that $c>0$ by Remark~\ref{prop:positivity}. The proof of Part (ii) is thus finished, and this completes the proof of Theorem~\ref{main-abstract-result}. 

 \subsection{Some useful criteria}

 We close this section with two useful criteria related to our assumptions $(S1)$ and $(S2)$. {Throughout this
 subsection, we always assume that (E1),(E2) and (I0)--(I3) hold.} We start with a criterion for a function $w$ to be contained in the set $\cM^+$ defined in (\ref{eq:defcMplus}).

 \begin{lem}\label{lem:criterionB3}
Let $w\in E^+ \setminus \{0\}$ satisfy  
   \begin{equation}  \label{eq:eigenvalue-alternative}
  \| v^-\| > \|w\| \qquad \text{for all } v \in E^0 \oplus E^- \text{ with } I(w+v)=0. 
\end{equation}
     Then $\sup_{E_w} \Phi<\infty$, i.e., $w \in \cM^+$. In particular, $(S2)$ holds if such a function $w$ exists.   
   \end{lem}
 \begin{proof}
   We argue by contradiction and assume that there exist $t_n>0$, $z^0_n \in E^0, z_n^- \in E^-$ with
   $\Phi(u_n) \to \infty$ for $u_n:= t_n w + z^0_n + z^-_n$. Setting $y_n^0:= \frac{z^0_n}{t_n}$ and $y_n^-:=
   \frac{z^-_n}{t_n}$, we thus note that 
   $$
    \Phi(u_n) = \frac{t_n^2}{2}(\|w\|^2- \|y_n^-\|^2) - I\bigl(t_n ( w +  y^0_n + y^-_n)\bigr)
   $$
  Since $I$ is nonnegative, we  must have $\|y_n^-\|<\|w\|$ and $t_n\to\infty$. 

  \medskip
  
  Set $C:=\inf_{n \in \N} t_n\|u_n\|^{-1}$ and we claim $C>0$.
  Indeed, otherwise we would have $\frac{u_n^+}{\|u_n\|} \to 0$ after passing to some subsequence and
  therefore $(u_n)_n$ is bounded by Lemma~\ref{lem:boundedness}(i), which contradicts the fact that $\Phi$ is bounded from above on bounded
  sets.  This implies that the sequence $(y^0_n)_n$ is bounded as well since 
$$    
\|y_n^0\| = \frac{\|z^0_n\|}{t_n} \le \frac{\|z^0_n\|}{C \|u_n\|} \le \frac{1}{C}.
$$
We may thus assume that $y_n^- \rightharpoonup y^-$ and $y_n^0 \rightharpoonup y^0$, $y:=y^0+y^-$. If  we had
  $I(w+ y)=0$,  then we would have $\|y_n^-\|\geq \|y^-\|+o(1) > \|w\|$ by \eqref{eq:eigenvalue-alternative} for
  large $n$, and therefore 
  $$
    \Phi(u_n)
    \le  \frac{t_n^2}{2}(\|w\|^2- \|y_n^-\|^2) < 0 \qquad (n\to\infty), 
  $$
  a contradiction. Hence $I(w+ y)>0$, and for any given $t>0$ we have, for $n$ sufficiently large, by the
  monotonicity condition in $(I0)$ and the weak lower semicontinuity in $(I3)$,
  \begin{align*}
    t_n^{-2} \Phi(u_n) 
    &= \frac{1}{2}(1- \|y_n^-\|^2) - t_n^{-2} I\bigl(t_n ( w +  y^0_n + y^-_n)\bigr) \\
    &\le \frac{1}{2}- t^{-2} I\bigl(t ( w +  y^0_n + y^-_n)\bigr) + o(1)\\
    &\leq \frac{1}{2}- t^{-2} I\bigl(t ( w +  y)\bigr) + o(1).
  \end{align*} 
  Moreover, since $I(w+ y)>0$, it follows from $(I0)$ that $t^{-2} I\bigl(t ( w +  y)\bigr)> 1$ for $t$ sufficiently large, which then implies that
  $$
  t_n^{-2} \Phi(u_n) \le - \frac{1}{2} + o(1) \qquad \text{as }n \to \infty,
  $$
  a contradiction. The proof is thus finished. 
 \end{proof}

Note that \eqref{eq:eigenvalue-alternative} holds in the simple situation where $I(u)=0,u\in E$ implies 
$u=0$. Indeed, in that case no $v\in E^0\oplus E^-$ as in \eqref{eq:eigenvalue-alternative} exists and there
is nothing to prove. Next, we add a criterion for condition $(S1)$. 

 \begin{prop} \label{prop:criterionB2I}
   Assume that for any given $u\in E$ with $I(u)>0$ we have 
   \begin{equation}\label{eq:criterionB2I}
     I(u)-I((1+s)u+v)+I'(u)[s(s/2+1)u+(1+s)v]\leq 0
   \end{equation}
   for all $v\in E^0\oplus   E^-$ and $s \ge -1$, with strict inequality whenever $v\in E^0$ and $su+v\neq 0$.
   Then $(S1)$ is satisfied.
  \end{prop}
  \begin{proof}
    We mimick the proof of \cite[Proposition 39]{SzuWet}.
    Let $w\in M^+$ and let $u$ denote a critical point of $\Phi$ on $\R^+ w \plus  E^0\plus E^-$.  We know
    from Proposition~\ref{prop:IPositivity} that this implies $I(u)>0$. 
    To prove that there are no other critical points it suffices to show that $u$ is a
    strict global maximizer\footnote{Then any other critical point would have to be another strict global
    maximizer, which is impossible.}. From $\Phi'(u)|_{\R w  \oplus E^0\oplus E^-}=0$ we infer
that $\Phi'(u)[s(s/2+1)u+(1+s)v]=0$ for all $s\geq -1$ and all $v\in E^0\oplus E^-$ and therefore 
  \begin{align*}
    \Phi((1+s)u+v)-\Phi(u)
    &= \Phi((1+s)u+v)-\Phi(u) -  \Phi'(u)[s(s/2+1)u+(1+s)v] \\
    &= - \frac{1}{2}\|v^-\|^2 + \Big[\,I(u)-I((1+s)u+v) +  I'(u)[s(s/2+1)u+(1+s)v]\,\Big]. 
  \end{align*}
   By  assumption both terms on the right hand side are nonpositive, which yields $\Phi((1+s)u+v)\leq\Phi(u)$
   for all ${s\geq -1},v\in E^0\oplus E^-$.   
   Moreover, if $\Phi((1+s)u+v)=\Phi(u)$ holds, then we necessarily have $\|v^-\|=0$ and thus $v\in E^0$. So 
   our assumption about the strict inequality implies $su+v=0$, hence $(1+s)u+v=u$. So $u$ is the unique
   global maximizer as required in (S1).
 \end{proof}


\section{Existence of ground state solutions to generalized nonlinear wave equations}
\label{sec:exist-ground-state}

In this section we shall prove our main abstract result on the existence of  ground state solutions to
generalized nonlinear wave equations as stated in Theorem~\ref{thm-general-hyperbolic-intro}. In fact, we   
consider the more general version
\begin{equation}
  \label{eq:nonlinear-wave-section}
\cA u + \partial_{tt}u = q(x,t)f(u),\qquad x \in M,\; t\in \mathbb{S}^1
\end{equation}
of (\ref{eq:nonlinear-wave-intro}) with a more general nonlinearity $f(u)$ in place of $|u|^{p-2}u$. Here, as
in the introduction, we let $(M,d,\mu)$ be a metric measure space and consider a self-adjoint operator $\cA:
D(\cA) \subset L^2(M) \to L^2(M)$ satisfying condition $(A)$. Moreover, we assume that $q \in L^\infty(M
\times \mathbb{S}^1)$.\footnote{We recall here that the spaces $L^p(M\times \mathbb{S}^1)$, $1 \le p \le \infty$ are defined
with respect to the product measure of $\mu$ with the one-dimensional Hausdorff measure on $\mathbb{S}^1$.}

As before, we let $\nu_0(\cA) \le \nu_1(\cA) \le \dots \le \nu_k(\cA) \le \dots \to +\infty$ be the sequence
of   eigenvalues of $\cA$ (counted with multiplicity), and we fix an associated orthonormal basis of $L^2(M)$
consisting of eigenfunctions $\zeta_k\in D(\cA)$, $k \in \N_0$. Moreover, we define
\begin{equation}
  \label{eq:def-e-l}
e_0(t) :={1}\qquad\; \text{and}\qquad\;  e_l(t):=\cos(lt),\quad e_{-l}(t):=\sin(lt)\qquad \text{for
$l\in\N$,}
\end{equation}
so $\{e_l:l\in\Z\}$ is an orthonormal basis of $L^2(\mathbb{S}^1)$. Here we identify functions on $\mathbb{S}^1$ with
$2\pi$-periodic functions on $\R$. Consequently, the functions 
$$
\zeta_k \otimes e_l \in L^2(M \times \mathbb{S}^1), \qquad [\zeta_k \otimes e_l](x,t) = \zeta_k(x)e_l(t),
\qquad k \in \N_0, l \in \Z 
$$
form an orthonormal basis of
$$
H:= L^2(M \times \mathbb{S}^1).
$$
The following basic result seems to be essentially known, but we could not find it in this particular form in the literature. 
\begin{prop}
  \label{prop-self-adjoint-realization}
  The generalized wave operator $L_{\cA}:=\cA+\partial_{tt}$, defined by
  $$
    D(L_{\cA}) := \Big\{v = \sum_{k\in\N_0,l\in\Z} a_{kl}\zeta_k \otimes e_l \in H \::\:  
  \sum_{k\in\N_0,l\in\Z} |a_{kl}|^2 (\nu_k-l^2)^2<\infty\Big\}.
$$
and
$$
L_{\cA}v = \sum_{k\in\N_0,l\in\Z} (\nu_k-l^2) a_{kl}\zeta_k \otimes e_l \qquad \text{for}\quad  v= \sum_{k\in\N_0,l\in\Z} a_{kl}\zeta_k \otimes e_l \quad \in \; D(L_{\cA})
$$
  is selfadjoint in $H$ with eigenpairs $(\nu_k-l^2,\zeta_k \otimes e_l)$ for $k\in\N_0,l\in\Z$.
\end{prop} 
\begin{proof}
Clearly, with the definitions above, the operator $L_{\cA}: D(L_{\cA}) \subset H \to H$ is symmetric, so we have $D(L_{\cA}) \subset D(L_{\cA}^*)$. To show its self-adjointness, let $v\in D(L_{\cA}^*) \subset H$. Then there is $C>0$ with 
\begin{equation}
  \label{eq:C-self-adjoint}
\skp{L_{\cA}u}{v} \le    C \|u\|_H\qquad \text{for all $u\in D(L_{\cA})$.}
\end{equation}
We can write $v$ as a convergent series  
$$
v = \sum_{k\in\N_0,l\in\Z} a_{kl}\zeta_k \otimes e_l \qquad \text{in $H$.}
$$
For fixed $M \in \N$, we then consider the function
$$
u = \sum_{\stackrel{k\in\N_0,l\in\Z}{k + |l| \le M}} a_{kl} \zeta_k \otimes e_l \qquad \in \;
  D(L_{\cA})
$$
An application of (\ref{eq:C-self-adjoint}) gives 
  $$
\sum_{\stackrel{k\in\N_0,l\in\Z}{k + |l| \le M}} (\lambda_k-l^2)^2  |a_{kl}|^2 =\skp{L_{\cA}u}{v} \le C \|u\|_H \le C \|v\|_H
  $$
  This inequality holds for all $M\in\N$, we conclude that 
$$
\sum_{k\in\N_0,l\in\Z} (\lambda_k-l^2)^2 |a_{kl}|^2 \le C  \|v\|_H < \infty,
$$
so $v \in D(L_{\cA})$. This proves $D(L_{\cA}^*)\subset D(L_{\cA})$, so $L_{\cA}$ is self-adjoint. 
\end{proof}

Based on Proposition~\ref{prop-self-adjoint-realization}, we may define
the spectral projections $P^\pm, P^0 \in \cL(H)$, the Banach space $(E,\|\cdot\|)$, the subspaces $E^\pm, E^0
\subset H$, the scalar products $\langle \cdot,\cdot \rangle_\pm$ and norms $\|\cdot\|_\pm$ on $E^\pm$ as in
the introduction. In the following, we write again $u^\pm$ in place of $P^\pm u$ for $u \in E$ and $u^0$ in
place of $P^0 u$.   

We then fix $p>2$, and we assume conditions $(CE)_p$ and $(CC)_q$ from the introduction. Moreover, we consider a nonlinearity $f \in C(\R)$ in (\ref{eq:nonlinear-wave-section}) satisfying the following assumption:
\begin{itemize}
 \item[$(MC)_p$] ({\em $p$-monotonicity condition})
   \begin{itemize}
   \item[$(f_1)$] The function $\frac{f(r)}{|r|}$ is strictly increasing on $\R$ with $\lim \limits_{r \to 0} \frac{f(r)}{|r|} = 0$.
   \item[$(f_2)$] There exists constants $c,C>0$ with 
        $$
         \int_0^s f(\tau)d\tau \geq c |s|^{p}\quad \text{and}\quad |f(s)|\leq C(1+|s|^{p-1}) \qquad \text{for all $s \in \R$.}
        $$
   \end{itemize}
\end{itemize}
By definition, a function $u \in E$ is a weak solution of (\ref{eq:nonlinear-wave-section}) if it satisfies
$$
\langle u^+,v^+ \rangle_+ -\langle u^-,v^- \rangle_- = \int_{M \times \mathbb{S}^1}q(x,t)f(u(x,t))v(x,t)\,d(x,t) \qquad 
\text{for all }v \in E.  
$$

The main result of this section is the following.
\begin{thm}
\label{thm-general-hyperbolic-section}
  Let $\cA$ be a selfadjoint operator in $L^2(M)$ satisfying the assumption (A).
  Suppose that $p>2$, $f \in C(\R)$ and $q \in L^\infty(M \times \mathbb{S}^1)$, $q \ge 0$ are chosen with the
  properties that $q>0$ on some open subset of $M \times \mathbb{S}^1$ and that conditions $(CE)_p$, $(CC)_q$ and $(MC)_p$ hold.
  Then (\ref{eq:nonlinear-wave-intro}) admits a ground state solution of (\ref{eq:nonlinear-wave-section}) on $M
  \times \mathbb{S}^1$.
\end{thm}

The remainder of this section is devoted to the proof of this theorem. For this we shall apply
Theorem~\ref{main-abstract-result} to the functional
\begin{equation*}
\Phi \in C^1(E), \qquad \Phi(u)=\frac{1}{2}\|u^+\|^2-\frac{1}{2}\|u^-\|^2-I(u) 
=\frac{1}{2}\|u^+\|_+^2-\frac{1}{2}\|u^-\|_-^2-I(u),
\end{equation*}
where
$$
I(u) = \int_{M \times \mathbb{S}^1}q(x,t)F(u(x,t))\,d(x,t)\qquad \text{with}\quad F(r)= \int_0^r f(\rho)\,d\rho.
$$
Using assumption $(CE)_p$ and the definition of the norm $\|\cdot\|$ on $E$, it is straightforward to
check that $I$ is   well-defined and of class $C^1$ on $E$ with 
$$
I'(u)v=  \int_{M \times \mathbb{S}^1}q(x,t)f(u(x,t))v(x,t)\,d(x,t) \qquad \text{for all }v \in E.  
$$ 
Consequently, critical points of $\Phi$ are precisely the weak solutions of
(\ref{eq:nonlinear-wave-section}).

In order to prove Theorem~\ref{thm-general-hyperbolic-section} by means of 
Theorem~\ref{main-abstract-result}, we only have to show that conditions $(MC)_p$, $(CE)_p$ and $(CC)_q$ imply
the abstract assumptions $(I0)-(I3)$ and $(S1)$, $(S2)$ from Section~\ref{sec:an-abstr-exist}. Note   that the
conditions $(E1)$, $(E2)$ from Section~\ref{sec:an-abstr-exist} already hold by the construction of the Banach
space $E$ explained in the Introduction.
In the following subsections, we verify assumptions $(I0)-(I3)$ and $(S1)$, $(S2)$ successively.

\subsection{Verification of condition $(I0)$.}

Since $q\geq 0$ and $F\geq 0$ we have $I\geq 0$. Moreover, since $F \in C^1(\R)$ with
$|F'(s)|= |f(s)|\leq C(1+|s|^{p-1})$ for all $s \in \R$ and $q \in L^\infty(M \times \mathbb{S}^1)$, a standard
argument shows that $I\in C^1(E)$. Next we deduce from  $(MC)_p$ the inequality   
$$
  F(r) = \int_0^r \frac{f(\rho)}{\rho} \rho \,d\rho \le \frac{f(r)}{r}\int_{0}^r \rho \,d\rho =
  \frac{f(r)r}{2} \qquad \text{for }r \in \R,
  $$
which implies that 
$$
\frac{d}{dr} \frac{F(r)}{r^2} = \frac{f(r)r -2F(r)}{r^3} 
$$
is nonnegative for $r>0$ and nonpositive for $r<0$. This ensures that, for every $u \in E$, the
function $r \mapsto \frac{I(ru)}{r^2}$ is nondecreasing on $(0,\infty)$. Next we note that $(MC)_p$
implies that for every $\eps>0$ there exists $C_\eps>0$ with
\begin{equation}
  \label{eq:F-eps-est}
  0\leq F(r) \le \eps |r|^2 + C_\eps |r|^{p} \qquad \text{for all $r \in \R$}.
\end{equation}
From this it is easy to deduce that 
$$
\lim_{r \to 0} \frac{I(ru)}{r^2}= 0 \qquad \text{for every }u \in E,
$$
and this implies, in particular, that $I(0)=I'(0)=0$.

To verify $\lim \limits_{r \to + \infty}\frac{I(ru)}{r^2}= +\infty$ if $I(u)>0$ let us assume the latter.
For sufficiently small $\eps>0$ depending on $u$ we then have 
$$
  I(u)-\eps \int_{M \times \mathbb{S}^1}q(x,t) |u|^2\,d(x,t)> 0.
$$
Then, by $(f_2)$ in $(MC)_p$  and \eqref{eq:F-eps-est},  
\begin{align*}
  \frac{I(ru)}{r^2}
  &=  \int_{M \times \mathbb{S}^1}q(x,t) \frac{F(ru)}{r^{2}}\,d(x,t) \ge  c r^{p-2} \int_{M \times \mathbb{S}^1}q(x,t)|u|^{p}\,d(x,t) \\
  &\ge  \frac{c r^{p-2}}{C_\eps} \int_{M \times \mathbb{S}^1}q(x,t) \big(F(u)-\eps |u|^2\big)\,d(x,t)= \frac{c r^{p-2}}{C_\eps} \Big(I(u)-\eps \int_{M \times \mathbb{S}^1}q(x,t) |u|^2\,d(x,t)\Big).  
\end{align*}
So the positivity of the last term gives the claim, and thus $(I0)$ is satisfied.

\subsection{Verification of condition $(I1)$.}

We need the following lemma.

\begin{lem}
\label{prelim-lemma-B5}  
For every $\delta>0$ there exists a constant $K_\delta>0$ with
$$
F(r+s) \le \delta + K_\delta\bigl(F(r)+F(s)\bigr) \qquad \text{for all }r,s \in \R.
$$
\end{lem}

\begin{proof}
  Let $\delta>0$ be arbitrary and choose $\eps=\eps(\delta)$ such that $|r+s|\leq \eps$ implies $F(r+s)\leq
  \delta$. So the desired estimate holds if $|r+s|\leq \eps$. In the 
  case $r+s\geq \eps$, set $R:=\max\{r,s\}$. From $r+s\geq \eps$ we deduce $R\geq \eps/2$ as well as
  $F(r+s)\leq F(2R)$ given that $F$ is nondecreasing on $(0,\infty)$. This implies, for some suitable positive
  $C_\eps'$ independent of $r$ and $s$, 
  $$
    F(r+s)
    \leq F(2R)
    \stackrel{\eqref{eq:F-eps-est}}\leq \eps (2R)^2 + C_\eps|2R|^p 
    \leq (\eps^{3-p}+C_\eps) |2R|^p
    \stackrel{(f_2)}\leq C_\eps' F(R)
    \leq \delta+ C_\eps'(F(r)+F(s)).
  $$
  The proves the desired estimate with $K_\delta:=C_\eps'$. The remaining case $r+s\leq -\eps$ is treated
  analogously.
\end{proof}

We now complete the verification of condition $(I1)$. Let $\eps>0$ be given. We recall that $q \in L^\infty(M
\times \mathbb{S}^1) \subset L^1(M \times \mathbb{S}^1)$ since $M$ has finite measure by assumption. Hence we
may apply Lemma~\ref{prelim-lemma-B5} with $\delta =  \eps \|q\|_{L^1(M \times \mathbb{S}^1)}^{-1}$. This
gives the estimate
\begin{align*}
  I(\phi+ \psi)
  &= \int_{M \times \mathbb{S}^1} q(x,t) F(\phi+\psi) \,d(x,t) \\
  &\leq  \int_{M \times \mathbb{S}^1} q(x,t)\Bigl( \delta + K_\delta \bigl(F(\phi)+ F(\psi)\bigr)\Bigr)\,d(x,t)\\
  &\leq  \delta \|q\|_{L^1(M \times \mathbb{S}^1)} + K_\delta
    \int_{M \times \mathbb{S}^1} q(x,t)\bigl(F(\phi)+ F(\psi)\bigr)\,d(x,t)\\
  &=\eps + K_\delta
    \int_{M \times \mathbb{S}^1} q(x,t)\bigl(F(\phi)+ F(\psi)\bigr)\,d(x,t)
  =\eps + K_\delta \bigl(I(\phi) + I(\psi)\bigr).
 \end{align*} 
 Hence we get the required inequality in $(I1)$ with $\kappa_\eps := K_\delta$.

\subsection{Verification of condition $(I2)$.} 
\label{sec:verif-cond-i2}
Let $u \in E$ and let $(w_n)_n$ be a sequence in $E^+$ with $w_n \weak w$ weakly in $E^+$ as $n \to \infty$. Then $w_n \to w$
strongly in $L^p(M \times \mathbb{S}^1)$ by $(CE)_p$. {Since $f=F'$ satisfies $(f_1),(f_2)$ we have}
\begin{align*}
  |I(u+w_n)-I(u+w)| 
  &\leq {C\|q\|_\infty \int_{M \times \mathbb{S}^1} |F(u+w_n)-F(u+w)| \,d(x,t)} \\
  &\leq C\|q\|_\infty \int_{M \times \mathbb{S}^1} { |w_n-w| \max\{|f(u+w_n)|,|f(u+w)|\} }\,d(x,t) \\
  &\leq C'\|q\|_\infty \int_{M \times \mathbb{S}^1} |w_n-w| ({1+}|u|^{p-1}+|w_n|^{p-1}+|w|^p) \,d(x,t) \\
  &\leq {C''}\|q\|_\infty \|w_n-w\|_{L^p(M \times \mathbb{S}^1)}
  ({1+}\|u\|_{L^p({M \times \mathbb{S}^1})}^{p-1}+\|w_n\|_{L^p({M \times \mathbb{S}^1})}^{p-1}+\|w\|_{L^p({M
  \times \mathbb{S}^1})}^{p-1})
  \\
  &\leq {C'''} \|w_n-w\|_{L^p({M \times \mathbb{S}^1})} = o(1) \qquad \text{as }n \to \infty.  
\end{align*}
This yields condition $(I2)$.

\subsection{Verification of condition $(I3)$.}
\label{sec:verif-cond-i3}

By assumption $(MC)_p$, the function $f$ is strictly increasing on $\R$ and therefore $F$ is strictly convex,
which implies the convexity of the functional $I$. Since $I$ is continuous, we conclude that $I$ is weakly
lower semicontinuous. Now let $(u_n)_n$ be a sequence in $E$ with $u_n \weak u$ in $E$ and $I(u_n)\to
I(u)$. Then 
$$
  \int_{M\times \mathbb{S}^1}f(u)(u_n-u)\,d\nu(x,t) \to 0 \quad \text{as }n \to \infty, \qquad \text{where }\,d\nu
  := q(x,t)\,d(x,t).
$$
This implies 
\begin{align*}
o(1)
= I(u_n)-I(u) 
&= \int_{M\times \mathbb{S}^1}\bigl(F(u_n)-F(u)\bigr)\,d\nu(x,t)\\ 
&= \int_{M\times \mathbb{S}^1} \bigl(F(u_n)-F(u)-f(u)(u_n-u)\bigr)\,d\nu(x,t) 
  \quad\text{as }n\to\infty.
\end{align*}
The convexity of $F$ implies that the integrand
$$
f_n:= F(u_n)-F(u)-f(u)(u_n-u)
$$
satisfies $f_n\geq 0$ and $f_n\to 0$ in $L^1(M\times \mathbb{S}^1, d\nu)$. By the Riesz-Fischer Theorem, we
can find a subsequence, again denoted by $(f_n)$, such that
the following holds $\nu$-almost everywhere: 
$$
  f_n\to 0 \text{ as }n\to\infty
  \qquad\text{and}\qquad
|f_n|\leq h  \text{ for all }n\in\N 
$$
for some function $h\in L^1(M\times \mathbb{S}^1,d\nu)$. Then the strict convexity of $F$ gives 
$$
u_n\to u \qquad \text{pointwise $\nu$-almost everywhere.}
$$
Moreover, by $(f_2)$ in $(MC)_p$ and Young's Inequality, we have 
\begin{align*}
 c|u_n|^p 
  &\leq  F(u_n) \leq f_n+|F(u)|+|f(u)||u-u_n| \\
  &\leq |h|+|F(u)|+|f(u)||u| + |f(u)||u_n| \\
  &\leq |h|+|F(u)|+|f(u)||u| + \frac{c}{2}|u_n|^p + C|f(u)|^{p'}   
\end{align*}
for some $C>0$ and thus  
$$
   |u_n|
   \leq \left(\frac{2}{c}\big(|h|+|F(u)|+|f(u)||u| + C|f(u)|^{p'}\big)\right)^{1/p}\qquad \nu-\text{a.e. for all } n. 
 $$ 
So the Dominated Convergence Theorem implies that 
\begin{align*}
  \|u_n-u\|_{L^p(M\times \mathbb{S}^1;d \nu)}^p
  = \int_{M\times \mathbb{S}^1} q(x,t)|u_n-u|^p \,d(x,t) 
  \to 0 \quad\text{as }n\to\infty.
\end{align*}
Thus finally implies
\begin{align*}
  \|u_n^0-u^0\| 
  &= \|u_n^0-u^0\|_{L^p(M\times \mathbb{S}^1;d \nu)}   \\ 
  &\leq  \|u_n-u\|_{L^p(M\times \mathbb{S}^1;d \nu)} +
  \|u_n^+-u^+\|_{L^p(M\times \mathbb{S}^1;d \nu)}  
  +\|u_n^--u^-\|_{L^p(M\times \mathbb{S}^1;d \nu)} \\
  &\leq  \|u_n-u\|_{L^p(M\times \mathbb{S}^1;d \nu)} +
  \|q\|_\infty\big( \|u_n^+-u^+\|_{L^p(M\times \mathbb{S}^1)}  
  +\|u_n^--u^-\|_{L^p(M\times \mathbb{S}^1)}\big) \\
  &\to 0 \qquad \text{as }n \to \infty.
\end{align*}
Here we used that $u_n^+\to u^+,u_n^-\to u^-$ in $L^p(M\times \mathbb{S}^1)$, which follows from  $(CE)_p$. So we
conclude that condition $(I3)$ holds.

\subsection{Verification of condition (S1).}  
\label{sec:verif-cond-s1}

We shall apply the abstract criterion given by Proposition~\ref{prop:criterionB2I} to verify condition $(S1)$. 
For this, we first note that condition $(f_1)$ in $(MC)_p$ implies that 
\begin{equation}
  \label{eq:szulkin-weth-lemma}
     F(u)-F((1+s)u+v) + F'(u)[s(s/2+1)u+(1+s)v] \leq 0
     \qquad\;\text{for } u,v\in\R,s\geq -1
\end{equation}
   with strict inequality whenever $(su,v)\neq (0,0)$, see \cite{SzuWet}[Lemma 38, Lemma~21].

 In order to apply Proposition~\ref{prop:criterionB2I}, we now let $u \in E$ with $I(u)>0$. For all $v\in E^0\oplus E^-$ and $s\geq -1$, we then have 
    \begin{align*}
      &I(u)-I((1+s)u+v)+I'(u)[s(s/2+1)u+(1+s)v] \\
      &=  \int_{M \times \mathbb{S}^1} q(x,t)\Big[ F(u)-F((1+s)u+v) + F'(u)[s(s/2+1)u+(1+s)v] \Big]\,d(x,t)
    \end{align*}    
  So $q\geq 0$ and \eqref{eq:szulkin-weth-lemma} give~\eqref{eq:criterionB2I}. Now it remains so show
  that equality in~\eqref{eq:criterionB2I} holds for some $v\in E^0$ only if 
  $su+v= 0$. To prove this let $Q:= \{(x,t) \in M \times \mathbb{S}^1\::\: q(x,t) \not = 0\}$. Then 
  $$
     F(u(x,t))-F((1+s)u(x,t)+v(x,t)) + F'(u(x,t))[s(s/2+1)u(x,t)+(1+s)v(x,t)]=0
  $$ 
  for almost all $(x,t) \in Q$. The strictness property of \eqref{eq:szulkin-weth-lemma} then implies
  $su(x,t)=v(x,t)=0$ for almost all $(x,t)\in Q$, in particular 
  $$
   \int_{M \times \mathbb{S}^1} q(x,t)|v(x,t)|^2\,d(x,t)= 0.
   $$
   As a consequence of condition $(CC)_q$, the latter implies $v\equiv 0$. On the other
   hand, the assumption $I(u)>0$ implies that $u$ does not vanish a.e. on $Q$, so $su(x,t)=0$ for almost all
   $(x,t)\in Q$ implies $s=0$. We thus conclude that $s=0$ and $v=0$, whence $su+v=0$.
   So Proposition~\ref{prop:criterionB2I} applies and (S1) holds.
    \smallskip

 We stress that the extra assumption $I(u)>0$ was necessary in Proposition~\ref{prop:criterionB2I} for the application in the present context. Without this assumption, we cannot deduce (S1) because of functions $u$ that are
  supported in $(M \times \mathbb{S}^1) \sm Q$ with $Q$ defined as above. Such functions satisfy $I(su)=0$ for all
  $s\in\R$, so strict inequality does not hold whenever $v\in E^0$ and $su+v\neq 0$ (choose $v=0$).
 
\subsection{Verification of condition (S2).}

To verify condition (S2), we need the following preliminary observation. 
 \begin{prop}
\label{existence-weighted-eigenvalue}
Let $U \subset M \times \mathbb{S}^1$ be an open subset. Then there exists $\phi \in D(L_{\cA})$ with
$\supp(\phi)\subset U$ and $\langle L_{\cA} \phi,\phi \rangle_H>0$.
 \end{prop} 
\begin{proof}
Let $N \subset M$ be an open subdomain and $I \subset \mathbb{S}^1$ be an open subinterval with $N \times I \subset U$. 
Moreover, let $\alpha \in C^2_c(I)$ be a nontrivial function, and let $w \in D(\cA),\supp(w)\subset N$. 
Then $(x,t) \mapsto \phi(x,t)= \alpha(t)w(x)$ defines a function in $D(L_{\cA})$ with 
$$
  \langle L_{\mathcal A} \phi, \phi \rangle_{H}
  = \langle L_{\mathcal A} \phi, \phi \rangle_{L^2(U)}
  = \langle \alpha'',\alpha \rangle_{L^2(I)} \|w\|_{L^2(N)}^2 + \langle \cA w,w \rangle_{L^2(N)}
\|\alpha\|_{L^2(I)}^2, 
$$
hence $\langle L_\cA \phi, \phi \rangle_{L^2(U)}>0$ if
$$
\frac{\langle \cA w,w \rangle_{L^2(N)}}{\|w\|_{L^2(N)}^2} > - \frac{\langle \alpha'',\alpha \rangle_{L^2(I)}}{\|\alpha\|_{L^2(I)}^2}.
$$
By assumption (A) we can find a function $w\in D(\cA)$ supported in $N$ such that this condition
is satisfied. Hence the claim is proved.
\end{proof}

Condition $(S2)$ is now a consequence of the following proposition.

\begin{prop}
\label{Mplus-criterion}
Assume that  $q(x,t) \geq 0$ with $q>0$ on an open subset $U \subset
M \times \mathbb{S}^1$. Then we have $\phi^+ \in \cM^+$ whenever $\phi \in D(L_{\cA})$ is chosen with
$\supp(\phi)\subset U$ and $\langle L_\cA \phi,\phi \rangle_H>0$.
\end{prop} 

 \begin{proof}
   We first note that
\begin{equation}\label{eq:phi_inequality}
  \|\phi^+\|^2 - \|\phi^-\|^2 = \langle L_{\cA}\phi,\phi \rangle_H>0,\quad
  \text{hence }  \|\phi^+\| > \|\phi^-\|\geq 0.
\end{equation}
We now apply the criterion given in Lemma~\ref{lem:criterionB3} to $w:=\phi^+$. So let $v\in E^0\oplus E^-$
with $I(w+v)=0$, which in particular implies that $w + v \equiv 0$ on $U$. Since 
$\phi\in D(\cA)$ implies $L_{\cA}\phi\in L^2(M\times \mathbb{S}^1)$, it then follows 
from $w^+=\phi^+,w^-=0$  and (A3) 
that 
\begin{align*}
  0 &= \langle L_{\cA}\phi, w + v \rangle_H 
  \stackrel{\eqref{eq:SKP_vs_operator}}= \langle \phi^+,(w+v)^+\rangle_+ - \langle
  \phi^-,(w+v)^-\rangle_-
  \\
  &= \|\phi^+\|^2 - \langle \phi^-,v^-\rangle,\quad\text{hence \eqref{eq:phi_inequality} implies }v^-\neq 0.
\end{align*}
So,  by Cauchy-Schwarz and  \eqref{eq:phi_inequality},
$$
\|\phi^+\|^2   = \langle \phi^-,v^-\rangle \le  \|\phi^-\| \|v^-\| < \|\phi^+\| \|v^-\|
$$
and hence
$$
  \|v^-\|>\|\phi^+\|= \|w\|. 
$$
So Lemma~\ref{lem:criterionB3} applies and yields $w\in\cM^+$ as claimed.
\end{proof}

\section{Compact Embeddings}
\label{sec:compact-embeddings}
 
In this section we verify the compactness condition $(CE)_p$ from Theorem~\ref{main-abstract-result} in the
special settings from the
Theorems~\ref{thm-wave-1+1},~\ref{thm-S-1-intro},~\ref{thm-T-N-intro},~\ref{thm-intro-sphere} and
\ref{thm-intro-sphere-klein-gordon}. We start with generalized nonlinear wave equations on tori from
Theorem~\ref{thm-T-N-intro} where the proof relies on the Hausdorff-Young inequality for Fourier series.

\subsection{Generalized wave equations on tori}

In the following assume $M=\mathbb T^N$ and $\cA=(-\Delta)^m$ where $m$ is even.
The generalized wave operator then reads 
$$
    L_{\mathcal A}:= \partial_{tt}+(-\Delta)^m
    \qquad\text{on}\quad \mathbb{T}^N \times \mathbb{S}^1.
  $$  
By Proposition~\ref{prop-self-adjoint-realization}, an orthonormal basis of $L^2(\mathbb{T}^N \times
\mathbb{S}^1)$ of eigenfunctions of $L_{\mathcal A}$ is given by
\begin{align*}
     \phi_{kl}(x,t) = c_0\, e_l(t)e_{k_1}(x_1) \cdot \ldots \cdot e_{k_N}(x_N)  
     \qquad \text{with }\quad l\in \Z, k\in\Z^N.
\end{align*}
Here $c_0$ is normalization constant, and the functions $e_l$ are given in \eqref{eq:def-e-l}. The
corresponding eigenvalues read 
$$
  \lambda_{kl}: = |k|^{2m}-l^2 = l_k^2-l^2,
    \qquad\text{where}\quad l_k:= |k|^m = (k_1^2+\ldots+k_N^2)^{\frac{m}{2}}   
$$
Since $m$ is even, $l_k$ is a natural number for all $k\in\Z^N$ and so 
$0$ is an eigenvalue of infinite multiplicity (the same holds for $N=1$, but this case will be covered by the
analysis for wave equations on spheres further below.)

  \begin{cor}  \label{cor-compact-embedding-torus}
     Assume $N\in\N,m\in\N$ and $M=\mathbb{T}^N,\cA=(-\Delta)^m$ where $m$ is even.
    Then the embedding $E^+\oplus E^- \hookrightarrow L^q(\mathbb{T}^N\times \bS^1)$ is compact provided that
    $1\leq q<\frac{2N}{(N-m)_+}$.
  \end{cor}
  \begin{proof}
    Since $\mathbb{T}^N\times \bS^1$ has finite measure, it suffices to prove the result for exponents
    $q>2$. Write $Z_*:= \{(k,l)\in\Z^N\times\Z: \lambda_{kl}\neq 0\}$ and, for $u\in E^+\oplus E^-$, 
    $$ 
      c_{kl}(u) := \int_{\T^N\times \bS^1} u(x,t)\phi_{kl}(x,t)\,d(x,t).
    $$
    The Hausdorff-Young inequality for $N$-dimensional Fourier series and H\"older's inequality give 
    \begin{align*}
      \|u\|_q
      &= \Big\| \sum_{(k,l)\in {\Z_*}} c_{kl}(u)\phi_{kl} \Big\|_q 
      \les \Big\| \big(c_{kl}(u)\big) \Big\|_{l^{q'}(Z_*)} \\
      &\les \Big\| \big(|\lambda_{kl}|^{1/2}c_{kl}(u)\big) \Big\|_{l^2(Z_*)} \Big\|
      \big(|\lambda_{kl}|^{-1/2}\big) \Big\|_{l^{\frac{2q}{q-2}}(Z_*)} \\      
      &\sim \|u\|_{E^+\oplus E^-} \Big(\sum_{(k,l)\in Z_*}
      |\lambda_{kl}|^{-\frac{q}{q-2}}\Big)^{\frac{q-2}{2q}}
  \end{align*}  
So it remains to show that the latter sum is finite. 
We shall use the following estimates for $k\neq 0$: 
 \begin{align*}
   &|(l_k+p)^2-l_k^2|
    = 2pl_k + p^2 \ge 2p l_k
    \gtrsim p |k|^m 
    &&\text{for }p\in\N, \\ 
  &|(l_k-p)^2-l_k^2| 
   =  p (2l_k-p)  
   \ge p l_k
   \gtrsim p|k|^m 
   &&\text{for }p=1,\ldots,l_k.
 \end{align*}    
 This gives
\begin{align*}
    \sum_{\lambda_{kl}\neq 0} |\lambda_{kl}|^{-\frac{q}{q-2}} 
    &\les \sum_{k\in\Z^N} \sum_{l\in\Z, l \not = \pm l_k} 
    |l^2-l_k^2|^{-\frac{q}{q-2}} \\
    &\les  \sum_{l=1}^\infty  (l^2)^{-\frac{q}{q-2}} 
    + \sum_{k\in\Z^N\sm\{0\}}    \left( \sum_{p=1}^\infty 
    (p|k|^m)^{-\frac{q}{q-2}} + \sum_{p=1}^{|k|^m} (p|k|^m)^{-\frac{q}{q-2}} \right)  \\
    &\sim  \sum_{l=1}^\infty l^{-\frac{2q}{q-2}} + 
    \sum_{k\in\Z^N\sm\{0\}}   |k|^{-\frac{mq}{q-2}} \sum_{p=1}^\infty  
    p^{-\frac{q}{q-2}}  \\
    &\sim 1 +  \sum_{k\in\N_0^N\sm\{0\}}   |k|^{-\frac{mq}{q-2}} 
    \;\sim \; 1 +   \sum_{r=1}^\infty   r^{-\frac{mq}{q-2}}\cdot r^{N-1}.
\end{align*}  
  The exponent is less than $-1$ if and only if  $q<\frac{2N}{(N-m)_+}$, and the continuity of the
  embedding is proved via the estimate $\|u\|_q\les \|u\|_{E^+\oplus E^-}$. The compactness follows by
  approximation by finite linear combinations of the orthonormal basis. 
\end{proof}

\subsection{Generalized wave equations on spheres}

We now focus on $M=\mathbb{S}^N$ and $\cA=(-\Delta_{\bS^N})^m$, which leads to the generalized
wave operator $$
    L_{\cA} := \partial_{tt}+(-\Delta_{\bS^N})^m
    \qquad\text{on}\quad \bS^N \times \mathbb{S}^1.
$$
In view of Proposition~\ref{prop-self-adjoint-realization} this operator admits an orthonormal basis in
$L^2(\bS^N\times \mathbb{S}^1)$ of eigenfunctions given by the functions
  \begin{equation}
    \label{eq:def-phi-k-l-i}
   \phi_{kli}(x,t) = c_0\,  e_l(t)Y_{k{i}}(x)
    \quad \text{with}\quad k\in\N_0, 1 \le {i} \le \rho_k,l\in\Z,  
  \end{equation}
  Here $c_0$ is normalization constant, the functions $e_l$ are given in \eqref{eq:def-e-l} and, for $k \in
  \N_0$, the functions $Y_{k1},\dots,Y_{k \rho_k}$ form an orthonormal basis in $L^2(\mathbb{S}^N)$ of the
  subspace of spherical harmonics of degree $k$, which is precisely the eigenspace of the Laplace-Beltrami
  operator $-\Delta_{\bS^N}$ associated with the $k$-th eigenvalue $\tilde \nu_k = k(k+N-1)$. 
   It is known that the dimension $\rho_k$ of this eigenspace is given by
  $\rho_k= N_k-N_{k-2}$ with $N_k=\binom{N+k}{k}$, but our analysis does not rely on the growth of $\rho_k$
  as $k\to\infty$.
  Since $(-\Delta_{\bS^N})^m Y_{k{i}}= \nu_k Y_{k{i}}$ with $\nu_k := \tilde \nu_k^m=k^m(k+N-1)^m$, the
  eigenvalues of $L_{\cA}$ associated with the functions $\phi_{kl{i}}$ are given by 
  $$
    \lambda_{k,l} := \nu_k-l^2 = l_k^2-l^2 \qquad\text{with}\qquad l_k:=k^{\frac{m}{2}}(k+N-1)^{\frac{m}{2}}.
  $$
  In order to single out the case where $L_{\cA}$ has an infinite-dimensional kernel we assume
  $$ 
    N=1\quad\text{or}\quad m \text{ is even}.
  $$
  Indeed, in this case we have $\lambda_{kl}=0 \Leftrightarrow |l|=l_k$ where $l_k$ is a natural number for
  all $k\in\N_0$.
  
  \medskip
  
To prove the embeddings of $E^+\oplus E^-$ we perform an eigenfunction expansion. 
To this end we introduce, for any given $w\in L^2(\bS^N)$, the
projection on the space of spherical harmonics of degree $k$ by the formula 
$$
  (\Pi_k w)(x) := \sum_{i=1}^{\rho_k} \skp{w}{Y_{ki}}_{L^2(\bS^N)} Y_{ki}(x).
$$
The eigenfunction expansion of the Laplace-Betrami operator on $\bS^N$ gives 
\begin{equation} \label{eq:EigenfunctionExpansionSN}
  w= \sum_{k=0}^\infty \Pi_k w \;\text{in }L^2(\bS^N)
  \quad\text{as well as}\quad  
  \sum_{k=0}^\infty \|\Pi_k u\|_{L^2(\bS^N)}^2 = \|u\|_{L^2(\bS^N)}^2.     
\end{equation} 
Sogge's bounds from \cite[Theorem~4.2]{Sogge} give
\begin{equation}\label{eq:Sogge_bound}
  \|\Pi_k w\|_{L^p(\bS^N)} \les (1+k)^{\sigma_p} \|\Pi_k w\|_{L^2(\bS^N)},
  \quad\text{where } 
  \sigma_p := \begin{cases}
    \frac{(N-1)(p-2)}{4p}   &,\text{if } 2\leq p\leq \frac{2(N+1)}{N-1} \\
    \frac{p(N-1)-2N}{2p}  &,\text{if } p\geq \frac{2(N+1)}{N-1} 
  \end{cases},
\end{equation}
see also \cite[Theorem~4]{Zhou_waveSn}. Note that the theorem is stated for $N\geq 2$, but it trivially
extends to the case $N=1$. 
It will turn out useful to choose ``mode shift'' $k_l\in\N_0$ onto
the most resonant spherical harmonic. Roughly speaking, we will choose $k_l$ in such a way that 
$|\nu_{k}-l^2|$ is smallest possible or close to smallest possible when $k=k_l$. Writing $u_l(x):=
\skp{u(x,\cdot)}{e_l}_{\bS^1}$ we obtain, for any given $u\in E^+\oplus E^-$, the formal expansion 
$$
  u(x,t) 
  = \sum_{l\in\Z} u_l(x) e_l(t)
  = \sum_{l\in\Z, k\in\N_0, \nu_k\neq l^2} \Pi_k(u_l)(x) e_l(t)
  = \sum_{j\in\Z} \sum_{l\in\Z_j}  \Pi_{k_l+j} (u_l)(x) e_l(t)
  = \sum_{j\in\Z} T_j(u)(x,t) 
$$
where 
$$
  T_j(u)(x,t) := \sum_{l\in\Z_j}  \Pi_{k_l+j} (u_l)(x) e_l(t),\qquad 
  \Z_j= \{l\in\Z: k_l+j\geq 0,\,\nu_{k_l+j}-l^2\neq 0\}.
$$
Here we used that contributions related to indices with $\nu_k=l^2$ are trivial for functions
in $E^+\oplus E^-$. 
 
\begin{prop} \label{prop:EmbeddingSn}
  For all $p\in (2,\infty]$ we have 
  $$
    \|u\|_{L^p(\bS^N\times \bS^1)} 
    \les   \Big(\sum_{j\in\Z} \Big( \sum_{l\in\Z_j}   |\nu_{k_l+j}-l^2|^{-\frac{p}{p-2}}
    (1+k_l+j)^{\frac{2p\sigma_p}{p-2}} \Big)^{\frac{p-2}{p}}  \Big)^{\frac{1}{2}}
    \|u\|_{E^+\oplus E^-}
  $$
  The embedding $E^+\oplus E^-\hookrightarrow L^p(\bS^N\times \bS^1)$ is compact if the series is convergent.
\end{prop}
\begin{proof}
  We first use the Hausdorff-Young inequality with respect to the time variable
  \begin{align*}
    \|T_j u\|_{L^p(\bS^N\times \bS^1)}
    &= \Big\|  \big\|\sum_{l\in\Z_j}  \Pi_{k_l+j} (u_l)(x) e_l(t)\big\|_{L_t^p(\bS^1)} \Big\|_{L^p_x(\bS^N)}
    \\
    &\les \Big\|  \big\|( \Pi_{k_l+j} (u_l)(x)) \big\|_{l^{p'}(\Z_j)}  \Big\|_{L^p_x(\bS^N)} \\
    &= \Big\|  \sum_{l\in\Z_j} |\Pi_{k_l+j} (u_l)|^{p'}  \Big\|_{L^{\frac{p}{p'}}(\bS^N)}^{\frac{1}{p'}}. 
  \end{align*}
  Then the triangle inequality, which applies due to $\frac{p}{p'}=p-1\geq 1$, and \eqref{eq:Sogge_bound} give
  \begin{align*}
    \|T_j u\|_{L^p(\bS^N\times \bS^1)}^{p'}
    &\les    \sum_{l\in\Z_j}   \| |\Pi_{k_l+j} (u_l) |^{p'} \|_{L^{\frac{p}{p'}}(\bS^N)} \\
    &=    \sum_{l\in\Z_j}   \| \Pi_{k_l+j} (u_l)  \|_{L^p (\bS^N)}^{p'} \\
    &\les    \sum_{l\in\Z_j}     (1+k_l+j)^{\sigma_p p'} \| \Pi_{k_l+j} (u_l) 
    \|_{L^2(\bS^N)}^{p'} \\
    &=   \Big\| \Big(  (1+k_l+j)^{\sigma_p} \| \Pi_{k_l+j} (u_l)\|_{L^2 (\bS^N)}\Big)
    \Big\|_{l^{p'}(\Z_j)}^{p'}.
  \end{align*}
  Next, H\"older's inequality with $\frac{1}{p'}=\frac{1}{2}+\frac{p-2}{2p}$ implies 
  ($\lambda_{k,l}=\nu_k-l^2$)
  \begin{align*}
    \|T_j u\|_{L^p(\bS^N\times \bS^1)}
    &\leq  \Big\| \Big(  |\lambda_{k_l+j,l}|^{\frac{1}{2}}\| \Pi_{k_l+j} (u_l)
    \|_{L^2(\bS^N)} \Big) \Big\|_{l^2(\Z_j)} \cdot 
    \Big\| \Big( |\lambda_{k_l+j,l}|^{-\frac{1}{2}} (1+k_l+j)^{\sigma_p} \Big)
    \Big\|_{l^{\frac{2p}{p-2}}(\Z_j)} \\
    &\leq \Big(  \sum_{l\in\Z_j}   |\lambda_{k_l+j,l}|\| \Pi_{k_l+j} (u_l) 
    \|_{L^2(\bS^N)}^2  \Big)^{\frac{1}{2}}  \cdot 
    \Big( \sum_{l\in\Z_j}  |\lambda_{k_l+j,l}|^{-\frac{p}{p-2}}
    (1+k_l+j)^{\frac{2p\sigma_p}{p-2}} \Big)^{\frac{p-2}{2p}}   \\
    &= \|T_j u\|_{E^+\oplus E^-}   
    \Big( \sum_{l\in\Z_j}  |\nu_{k_l+j}-l^2|^{-\frac{p}{p-2}} (1+k_l+j)^{\frac{2p\sigma_p}{p-2}}
    \Big)^{\frac{p-2}{2p}}.   
  \end{align*}
   From \eqref{eq:EigenfunctionExpansionSN} we get  
  $$
    \sum_{j=0}^\infty \|T_j u\|_{E^+\oplus E^-}^2 = \|u\|_{E^+\oplus E^-}^2,   
  $$
  so $u=\sum_{j=0}^\infty T_ju$ implies
  \begin{align*}
    \| u\|_{L^p(\bS^N\times \bS^1)}
    &\leq \sum_{j=0}^\infty \|T_j u\|_{L^p(\bS^N\times \bS^1)} \\
    &\les \sum_{j=0}^\infty \|T_j u\|_{E^+\oplus E^-}   
    \Big( \sum_{l\in\Z_j}  |\nu_{k_l+j}-l^2|^{-\frac{p}{p-2}} (1+k_l+j)^{\frac{2p\sigma_p}{p-2}}
    \Big)^{\frac{p-2}{2p}} \\
    &\les \| u\|_{E^+\oplus E^-} \Big(\sum_{j=0}^\infty
    \Big( \sum_{l\in\Z_j}  |\nu_{k_l+j}-l^2|^{-\frac{p}{p-2}} (1+k_l+j)^{\frac{2p\sigma_p}{p-2}}
    \Big)^{\frac{p-2}{p}}  \Big)^{\frac{1}{2}}.
  \end{align*}
  This proves  the estimate and, as above, the compactness  of the embedding.
\end{proof}

  \begin{cor}  \label{cor-compact-embedding-m-even-S-N}
    Assume  $M=\mathbb{S}^N$ and $\cA=(-\Delta_{\bS^N})^m$ 
    with $N=1$ or $m\in\N$ even. Then the embedding $E^+\oplus E^- \hookrightarrow L^p(M \times \mathbb{S}^1)$
    is compact provided that $1\leq p< \frac{2(N+1)}{(N-m)_+}$. 
  \end{cor}
\begin{proof} 
   We have $\nu_k = k^m(k+N-1)^m$, so choose 
$$
  k_l := [k_l^*] \quad\text{where } k_l^*:= - \frac{N-1}{2} + \sqrt{|l|^{\frac{2}{m}}+
  \big(\frac{N-1}{2}\big)^2} \geq 0 
$$
In other words, $k_l^*$ is the unique nonnegative real solution of $k^m(k+N-1)^m = l^2$ and $k_l\in\N_0$ is a
closest integer to $k_l^*$. We recall 
$$
  \Z_j
  = \{l\in\Z : k_l+j\in\N_0, \nu_{k_l+j}-l^2\neq 0\}
  = \{l\in\Z : k_l+j\in\N_0, k_l+j\neq k_l^*\}.
$$
Define $j^*\in\R$ via $k=k_l+j=k_l^*+j^* \in\N_0$  
and $A_{lj}:=j^*(j^*+2k_l^*+N-1)$. Then
\begin{align*}
  \nu_k-l^2
   &= k^{m}(k+N-1)^{m}-l^2 \\
   &=   (k^2+(N-1)k)^{m}-|l|^2 \\
   &=   \Big((k_l^*)^2+(N-1)k_l^*+j^*(j^*+2k_l^*+N-1)\Big)^{m}-|l|^2 \\
   &=   \Big(|l|^{\frac{2}{m}}+A_{lj}\Big)^{m}-|l|^2 \\
   &=   |l|^2 \Big( \Big(1+|l|^{-\frac{2}{m}}A_{lj}\Big)^{m}-1\Big).      
\end{align*}
From $k_l^*+j^*=k\geq 0$ and $(k_l^*,j^*)\neq (0,0)$ for $l\in\Z_j$ we infer
$$
  |A_{lj}|
  = |j^*| (j^*+2k_l^*+N-1)
  \sim |j^*| \max\{k_l^*,j^*\}
  \sim |j^*| (|l|+j_+^m)^{\frac{1}{m}}. 
$$
This implies $1+k_l+j \les |l|^{\frac{1}{m}}+ j_+$ and  
\begin{align} \label{eq:gap_estimate}
  \begin{aligned}
  |\nu_{k_l+j}-l^2|
   &=   |l|^2 \Big(  |l|^{-\frac{2}{m}}|A_{lj}| + (|l|^{-\frac{2}{m}}|A_{lj}|)^m\Big)\\
   &\sim \begin{cases}
        |l|^{\frac{2m-2}{m}}|A_{lj}| &, \text{if }|A_{lj}|\les |l|^{\frac{2}{m}}  \\
        |A_{lj}|^m &, \text{if }|A_{lj}|\ges |l|^{\frac{2}{m}}
   \end{cases} \\
   &\sim  \begin{cases}
        |l|^{\frac{2m-1}{m}}|j^*| &, \text{if }|j|\les |l|^{\frac{1}{m}}  \\
        j^{2m} &, \text{if }j\ges |l|^{\frac{1}{m}},
   \end{cases}           
\end{aligned}
\end{align}
Note that $j+k_l\geq 0$ implies that all relevant  cases are covered.

\medskip

\textbf{1st case: $j\neq 0$.}\; We then have $|j-j^*|=|k_l-k_l^*|\leq \frac{1}{2}$ and thus 
$\frac{1}{2}|j|\leq |j^*|\leq \frac{3}{2}|j|$. Note that this is false for $j=0$. So the above estimate gives
\begin{align*}
  &\sum_{l\in\Z_j}  |\nu_{k_l+j}-l^2|^{-\frac{p}{p-2}}
  (1+|l|^{\frac{1}{m}}+j_+)^{\frac{2p\sigma_p}{p-2}} \\
  &\sim \ind_{j\geq 1}  \sum_{l\in {\Z_j}, |l|\leq j^m} j^{-\frac{2mp}{p-2}+\frac{2p\sigma_p}{p-2}} +
    \sum_{l\in\Z_j, |l|\geq |j|^m} (|l|^{\frac{2m-1}{m}} |j^*|)^{-\frac{p}{p-2}}
  |l|^{\frac{2p\sigma_p}{m(p-2)}} 
    \\
  &\sim \ind_{j\geq 1}    j^{m-\frac{2mp}{p-2}+\frac{2p\sigma_p}{p-2}} +
   |j|^{-\frac{p}{p-2}}\cdot (|j|^m)^{1-\frac{(2m-1)p}{m(p-2)}+\frac{2p\sigma_p}{m(p-2)}} 
   \\ 
   &\sim   |j|^{m-\frac{2mp}{p-2}+\frac{2p\sigma_p}{p-2}}.  
\end{align*}  
So
\begin{align*}
  \sum_{|j|\geq 1} \Big( \sum_{l\in\Z_j} |\nu_{k_l+j}-l^2|^{-\frac{p}{p-2}}
  (1+|l|^{1/m}+j_+)^{\frac{2p\sigma_p}{p-2}}  \Big)^{\frac{p-2}{p}} 
  \sim \sum_{j=1}^\infty  j^{\frac{m(p-2)}{p}-2m+ 2\sigma_p}
\end{align*}
Since we are aiming for an embedding into $L^p(\bS^N\times\bS^1)$ with $p>\frac{2(N+1)}{N-1}$, we 
apply the criterion from Proposition~\ref{prop:EmbeddingSn} with $\sigma_p = \frac{p(N-1)-2N}{2p}$,
see~\eqref{eq:Sogge_bound}.
A short computation reveals that the series converges if and only if 
$(N-m)p< 2(N+m)$, which is true under the assumptions of the theorem. So it remains to deal with the
case $j=0$.

\medskip

\textbf{2nd case: $j=0$.}\; Now $j^*=k_l-k_l^*$ satisfies $0<|j^*|\leq \frac{1}{2}$, but $|j^*|$
is not  bounded away from zero any more. We have to prove that, under our assumptions, the sum
$$
  \sum_{l\in\Z_0} |\nu_{k_l}-l^2|^{-\frac{p}{p-2}}
  (1+|l|^{\frac{1}{m}})^{\frac{2p\sigma_p}{p-2}} 
$$
is finite. In view of \eqref{eq:gap_estimate} and $0\notin \Z_0$,
this is equivalent to proving 
$$
  \sum_{l\in\Z_0} (1+|j^*||l|^{\frac{2m-1}{m}})^{-\frac{p}{p-2}} |l|^{\frac{2p \sigma_p}{m(p-2)}}
  < \infty.
$$
So we study the asymptotics of $|j^*|=|k_l-k_l^*|$ as $l\to\infty$. To this end we make the ansatz $l=r^m+s$ 
$r\in\N,s\in\Z$ with $\un{s}(r)\leq s\leq \ov{s}(r)$ and $\un{s}(r)\in\Z_{\leq 0},\ov{s}(r)\in\Z_{\geq 0}$ 
defined via 
$$
  r-\frac{1}{2}\leq \sqrt{(r^m+s)^{\frac{2}{m}}+\frac{(N-1)^2}{4}}< r+\frac{1}{2} 
  \quad\Leftrightarrow\quad l=r^m+s, \un{s}(r)\leq s\leq \ov{s}(r).
$$ 
Solving these inequalities for $s$ one finds $|\un{s}(r)|,|\ov{s}(r)| = \frac{m}{2}r^{m-1} + O(r^{m-2})$ as $r\to\infty$.
For such $l,r,s$ we get
\begin{align*}
   |t^*|
   &=|k_l^*-k_l| = |k_l^*-[k_l^*]| \\
   &=  \Big|\sqrt{{|r^m+s|}^{\frac{2}{m}}+\frac{(N-1)^2}{4}}
   -\Big[\sqrt{{|r^m+s|}^{\frac{2}{m}}+\frac{(N-1)^2}{4}}\Big]\Big|  \\
   &= \Big| \sqrt{{|r^m+s|}^{\frac{2}{m}}+\frac{(N-1)^2}{4}} -r\Big| \\
   &= r \Big| \sqrt{{|1+sr^{-m}|}^{\frac{2}{m}}+\frac{(N-1)^2}{4r^2}} -1\Big|.       
\end{align*}
We define $s^*(r)\in \{\un{s}(r),\ldots,\ov{s}(r)\}$ as the uniquely determined resonant $s$-value where the
remainder term $t^*$ is exactly zero. This means 
$$
  \sqrt{(1+r^{-m}s^*(r))^{\frac{2}{m}}+\frac{(N-1)^2}{4r^2}} = 1
$$ 
or, more explicitly,
$$
  s^*(r)= -r^m+\Big(r^2-\frac{(N-1)^2}{4}\Big)^{\frac{m}{2}}.
$$
Then $l=r^m+s^*(r)$ satisfies $0=|j^*|=|k_l^*-k_l|$, so $l\in\Z_0$. As a consequence, we have $\nu_{k_l}\neq
l^2$ and the above ansatz leads to $l=r^m+s$ with $|\tilde s|\geq 1$ where $\tilde s:=s-s^*(r)$. 
We thus find for large $r$
 \begin{align*}
   |j^*|
   &\sim r \Big|\sqrt{\big(1+sr^{-m}\big)^{\frac{2}{m}}+\frac{(N-1)^2}{4r^2}} -1\Big| \\
   &\sim r \Big| \big(1+sr^{-m}\big)^{\frac{2}{m}}+\frac{(N-1)^2}{4r^2}  -1\Big| \\
   &= r \Big| \Big(1+s^*(r)r^{-m}+\tilde s r^{-m}\Big)^{\frac{2}{m}}+\frac{(N-1)^2}{4r^2} -1\Big| \\
   &= r \Big| \Big( \Big(1-\frac{(N-1)^2}{4r^2}\Big)^{\frac{m}{2}}+\tilde s
   r^{-m}\Big)^{\frac{2}{m}}+\frac{(N-1)^2}{4r^2} -1\Big| \\
   &= r \Big(1-\frac{(N-1)^2}{4r^2}\Big) \Big| \Big(1 +\tilde s
   r^{-m}\Big(1-\frac{(N-1)^2}{4r^2}\Big)^{-\frac{m}{2}}\Big)^{\frac{2}{m}}-1\Big| \\
   &\sim r \Big| \tilde s
   r^{-m}\Big(1-\frac{(N-1)^2}{4r^2}\Big)^{-\frac{m}{2}}\Big| \\     
   &\sim r^{1-m} |\tilde s|.    
 \end{align*}
 Note that the second last estimate holds because $\tilde s r^{-m}$ tends to zero uniformly as $r\to\infty$
 due to $\tilde s=O(r^{m-1})$.
 We conclude  with \eqref{eq:gap_estimate}, for $l=r^m+s \in \Z_0$,   
 \begin{align*}
   |\nu_{k_l}-l^2|
   \sim |l|^{\frac{2m-1}{m}}|j^*|
   \sim  r^{2m-1} r^{1-m}|\tilde s|  
   \sim  r^m |\tilde s|.     
 \end{align*}
This leads to  
 \begin{align*}
  \sum_{l\in\Z_0}  |\nu_{k_l}-l^2|^{-\frac{p}{p-2}}
  |l|^{\frac{2p\sigma_p}{m(p-2)}}  
  &\les 1+
  \sum_{r=r_0}^\infty \sum_{ s\in \{\un{s}(r),\ldots,\ov{s}(r)\}\sm\{0\}} 
   (r^m|\tilde s|)^{-\frac{p}{p-2}} 
  r^{\frac{2p\sigma_p}{p-2}}  \\
  &\sim 1 +  \sum_{r=r_0}^\infty r^{-\frac{mp}{p-2}+\frac{2p\sigma_p}{p-2}} \sum_{\tilde
  s\in\Z\sm\{0\}}  |\tilde s|^{-\frac{p}{p-2}}   \\
  &\sim 1 +  \sum_{r=r_0}^\infty r^{-\frac{mp}{p-2}+\frac{2p\sigma_p}{p-2}}.   
 \end{align*}
 Plugging in $\sigma_p=\frac{p(N-1)-2N}{2p}$ we find that the series converges
 if and only if $p<\frac{2(N+1)}{(N-m)_+}$, which finishes the proof. 
\end{proof}

The corresponding analysis related to the Klein-Gordon equation on spheres follows the same lines, but is
technically much simpler. The result is actually known, see \cite[Theorem~1]{Zhou_waveSn}. We present the
short proof for completeness.  
    
  \begin{cor}  \label{cor-compact-embedding-N-odd-S-N}
      Assume  $M=\mathbb{S}^N,\cA=-\Delta_{\bS^N} + c_N$ where $N\in\N$ is odd.
    Then the embedding $E^+\oplus E^- \hookrightarrow L^p({M\times\mathbb{S}^1})$ is compact provided that
    $1\leq p<\frac{2(N+1)}{N-1}$.
  \end{cor}
  \begin{proof}
   We have $\nu_k = k(k+N-1)+c_N = (k+\frac{N-1}{2})^2$, so choose $k_l:= (|l|-\frac{N-1}{2})_+$. 
   Then $\Z_0 =  \{l\in\Z : |l|<\frac{N-1}{2}\}$ is finite and so the corresponding sum  
   $$ 
     \sum_{l\in\Z_0}  |\nu_{k_l}-l^2|^{-\frac{p}{p-2}} (1+k_l)^{\frac{N-1}{2}}   
   $$
   is trivially finite. So we may concentrate on $|j|\geq 1$ where $\Z_j\subset \{l\in\Z: |l|+j\geq 0\}$.
    So any $l\in\Z_j$ satisfies $|l|+j\geq 0$ and hence 
   $$
     |\nu_{k_l+j}-l^2|
     = |(|l|+j)^2-l^2|
     = |2j|l|+j^2|
     = |j| |2|l|+j|
     \ges |j| (|l|+j_+).  
   $$
   Moreover, $1+k_l+j \leq 1+|l|+j$. So  
   \begin{align*}
     \sum_{l\in\Z_j}  |\nu_{k_l+j}-l^2|^{-\frac{p}{p-2}} (1+k_l+j)^{\frac{N-1}{2}}
     &\les \sum_{l\in\Z_j}  |j(|l|+j_+)|^{-\frac{p}{p-2}} (1+|l|+j)^{\frac{N-1}{2}} \\
     &\les |j|^{-\frac{p}{p-2}} \begin{cases}
       \sum_{l=0}^\infty  (l+j)^{-\frac{p}{p-2}}(1+l+j)^{\frac{N-1}{2}} &, \text{if }j\geq 1,\\
       \sum_{l=|j|}^\infty  l^{-\frac{p}{p-2}} (1+l-|j|)^{\frac{N-1}{2}} &, \text{if
       }j\leq -1 \end{cases} \\
     &\sim |j|^{-\frac{p}{p-2}} \cdot
       |j|^{1-\frac{p}{p-2}+\frac{N-1}{2}}   \\
     &\sim |j|^{-\frac{2p}{p-2}+\frac{N+1}{2}}.
   \end{align*}
   Here the second last estimate uses $-\frac{p}{p-2}+\frac{N-1}{2}<-1$.  
   This implies
   \begin{align*}
     &\;\sum_{j\in\Z\sm\{0\}} 
    \Big( \sum_{l\in\Z_j}  |\nu_{k_l+j}-l^2|^{-\frac{p}{p-2}} (1+k_l+j)^{\frac{N-1}{2}}
    \Big)^{\frac{p-2}{p}} \\
    &\les \sum_{j\in\Z\sm\{0\}}  \big( |j|^{-\frac{2p}{p-2}+\frac{N+1}{2}} \big)^{\frac{p-2}{p}} 
    \sim \sum_{j=1}^\infty  j^{-2 +\frac{(N+1)(p-2)}{2p}}.    
   \end{align*}
   Again, our assumption on $p$ implies that the exponent is less than $-1$. So the series converges and the
   claim follows from Proposition~\ref{prop:EmbeddingSn}.
\end{proof}
  
\section{On the control condition}
\label{sec:control-property}

The purpose of this section is to verify the control condition $(CC)_q$ from the introduction in the specific
settings associated with Theorems~\ref{thm-S-1-intro},~\ref{thm-T-N-intro},~\ref{thm-intro-sphere} and
\ref{thm-intro-sphere-klein-gordon}.
So let $(M,d,\mu)$ be a metric measure space and consider a self-adjoint
operator $\cA: D(\cA) \subset L^2(M) \to L^2(M)$ satisfying condition $(A)$ from the introduction. Moreover,
we consider the self-adjoint operator $L_{\cA}= \cA + \partial_{tt}$ on $L^2(M \times \mathbb{S}^1)$ given by
Proposition~\ref{prop-self-adjoint-realization}.

\medskip

\begin{defn} ~
  \begin{itemize}
  \item[(i)] We say that a subdomain $\Omega \subset M \times \mathbb{S}^1$ has the {\em control property} for the operator $L_\cA$ if, with some constant $C=C(\Omega)>0$, we have 
    \begin{equation*}
      \|u\|_{L^2(M \times \mathbb{S}^1)} \le C \|u\|_{L^2(\Omega)} \qquad \text{for all $u \in \ker(L_\cA)$.}
    \end{equation*}
\item[(ii)] We say that a subdomain $\omega \subset M$ has the {\em control property} for the operator $\cA$ if, with some constant $C=C(\omega)>0$, we have 
    \begin{equation}
      \label{eq:cp-space}
      \|u\|_{L^2(M)} \le C \|u\|_{L^2(\omega)} \qquad \text{for every eigenfunction $u \in D(\cA)$ of $\cA$.}
      \end{equation}
\end{itemize}
\end{defn}
Here, as before, $\ker(L_\cA)\subset  D(L_{\cA}) \subset L^2(M \times \mathbb{S}^1)$ denotes the kernel of the operator $L_\cA$.

Part (i) is clearly related to condition $(CC)_q$, as stated in the following remark.

\begin{rem}
  \label{rem-control-space-time}{\rm
(i) If $q$ is bounded below by a positive constant $c_q$ on a subdomain $\Omega \subset M \times \mathbb{S}^1$ satisfying the control property for the operator $L_\cA$, then the control condition $(CC)_q$ from the introduction is satisfied with $C= \frac{1}{c_q}$.\\
(ii) Clearly, if a subdomain $\omega \subset M$ has the {\em control property} for the operator    $\cA$ in
$M$, then it also has the control property for any operator of the form $P(\cA)$ in $M$, where $P$  is
polynomial with real coefficients which is strictly increasing on the spectrum $\sigma(\cA)$. Indeed, in this case, $P(\cA)$ also satisfies assumption $(A)$ from the introduction, and the eigenfunctions of $P(\cA)$ are
precisely the eigenfunctions of $\cA$. We denote the set of polynomials with this property by $\mathcal P_{\cA}$. }
\end{rem}

The following observation relates the control properties for the operators $\cA$ and $L_{\cA}$.  We recall
that, with the notation from the introduction, every function $u \in \ker(L_\cA)$ can be written as  
\begin{equation}\label{eq:tildeu}
  u(x,t) = \sum_{\nu_k=l^2} \alpha_{kl} \zeta_k(x)e_l(t)
  \quad\text{where}\quad \sum_{\nu_k=l^2} |\alpha_{kl}|^2 < \infty.
\end{equation}

\begin{lem}
\label{control-condition-implication}
  If $\omega \subset M$ satisfies the control property for the operator $\cA$,
then the set $\Omega = \omega \times \mathbb{S}^1$ satisfies the control property for the operator $L_{\cA}$.
\end{lem}
\begin{proof}
{Consider a function $u \in \ker(L_\cA)$ as in \eqref{eq:tildeu}.}   By orthogonality of the
functions $e_l$ in $L^2(\mathbb{S}^1)$, Fubini's theorem and the assumed control property of $\omega \subset M$ for the operator $\cA$, we then have
\begin{align*}
\|{u}\|_{L^2(M \times \mathbb{S}^1)}^2 
&= \sum_{\nu_k=l^2}  |\alpha_{kl}|^2 \|\zeta_k\|_{L^2(M)}^2 
\les  \sum_{\nu_k=l^2}  |\alpha_{kl}|^2 \|\zeta_k\|_{L^2(\omega)}^2 \\
&=  \int_{\omega}\Bigl(\sum_{\nu_k=l^2}  |\alpha_{kl}|^2 |\zeta_k(x)|^2\Bigr)\,dx 
=  \int_{\omega} \Bigl( \sum_{\nu_k=l^2}\sum_{\nu_i=j^2} \alpha_{kl}\zeta_k(x)
\overline{\alpha_{ij} \zeta_i(x)}\int_{\mathbb{S}^1}e_{l}(t) \overline{e_j(t)}dt \Bigr)\,dx\\
&=  \int_{\omega \times \mathbb{S}^1}\Bigl(\sum_{\nu_k=l^2}  \alpha_{kl} e_l(t)\zeta_k
(x)\Bigr)\overline{\Bigl(\sum_{\nu_i = j^2}  \alpha_{ij}
e_j(t)\zeta_i (x)\Bigr)}\,dx\,dt
=  \|{u}\|_{L^2(\Omega)}^2.   
\end{align*}
This proves the control property of $\Omega$ for the operator $L_{\cA}$.
\end{proof}

In the following, we consider the special case where $(M,g)$ is a closed Riemannian manifold and $\cA: D(\cA)
\subset L^2(M) \to L^2(M)$ is a selfadjoint partial differential operator on $M$ with compact resolvent.
As before, we consider the self adjoint (unbounded) generalized wave operator $L_\cA:= \cA + \partial_{tt}$
in the Hilbert space $L^2(M \times \mathbb{S}^1)$.  In this setting, we have the following
theorem, which follows by combining a result of Bardos, Lebeau and Rauch \cite{bardos-lebeau-rauch-1992} with
an argument of Komornik \cite{komornik-1992}.

\begin{thm}
  \label{closed-geodesic}
Let $\omega \subset M$ be an open set with the property that there exists $T>0$ such that every geodesic in
$M$ parametrized by arclength on the interval $[0,T]$ meets $\omega$.  Then $\omega$ satisfies the control
property for the operator $\cA = -\Delta$ in $M$.
\end{thm}

\begin{proof}
Let $T>0$ be chosen as in the assumption. By the main result in \cite{bardos-lebeau-rauch-1992} of Bardos, Lebeau and Rauch and the Hilbert
uniqueness method, we know that for every classical (complex-valued) solution $u$ of the equation
$\partial_{tt}u-\Delta u = 0$ we have the estimate
\begin{equation}
  \label{eq:BLR-estimate}
\|u(0,\cdot)\|_{L^2(M)}^2 \le C_T \int_0^T \|u(t,\cdot)\|_{L^2(\omega)}^2\,dt
\end{equation}
with a constant $C_T>0$ which does not depend on $u$. 
So if $v$ is an eigenfunction of $-\Delta$ in $M$ corresponding to the eigenvalue $\lambda$, we may apply (\ref{eq:BLR-estimate}) to the function $(t,x) \mapsto u(t,x)= e^{i\sqrt{\lambda}t}v(x)$ which is a classical solution of $\partial_{tt}u-\Delta u = 0$ on $[0,T]$. This gives
$$
\|v\|_{L^2(M)}^2 \le T C_T \|v\|_{L^2(\omega)}^2,
$$
so (\ref{eq:cp-space}) holds with $C = \sqrt{T C_T}$. 
\end{proof}

\begin{cor}
  \label{cor-CC-sphere}
  Let $M= \mathbb{S}^N$, let $\cA = P(-\Delta)$ for some $P\in\mathcal P_{\cA}$, and let $\omega \subset
  \mathbb{S}^N$ be an open subset which intersects every great circle on $\mathbb{S}^N$. 
  Then $\omega$ has the control property for the operator  $\cA$.  
\end{cor}
\begin{proof}
  Since every geodesic of length greater than or equal to $2\pi$ on $\mathbb{S}^N$ contains a full great circle, the set $\omega$ satisfies the assumptions of Theorem~\ref{closed-geodesic} with $T= 2\pi$.
  Then, by Theorem~\ref{closed-geodesic}, the set $\omega$ satisfies the control property for the operator
  $-\Delta$ in $M$. By Remark~\ref{rem-control-space-time}, it then also has the control property for $\cA=
  P(-\Delta)$.
\end{proof}

We collect the following further results related to the flat torus $M=\mathbb T^N$. 

\begin{thm} \label{thm-komornik}  ~
  \begin{itemize}
  \item[(i)] Every nonempty open subset $\Omega \subset \mathbb T^N \times \mathbb{S}^1$ has the control
  property for the operator $L_{\cA}$ with $\cA= (-\Delta)^2$.
  \item[(ii)] Every nonempty open subset $\Omega \subset \mathbb T^1 \times \mathbb{S}^1$ has the control
  property for the operator $L_{\cA}$ with $\cA= (-\Delta)^m$ for some $m \ge 2$.
  \item[(iii)] Every nonempty open subset $\omega \subset \mathbb T^N$ has the control property for the
  operator $\cA=P(-\Delta)$ whenever $P\in\mathcal P_{-\Delta}$.  As a consequence, $\omega\times\bS^1$
  has the control property for $L_{\cA}$.
  \end{itemize}  
\end{thm}

Parts (i) and (iii) are due to Komornik \cite{komornik-1992}, although the results are stated in a somewhat different form in \cite{komornik-1992}. Part (ii) follows by similar arguments as in \cite{komornik-1992}. For the convenience of the reader, we include the proof of these facts here. 

\begin{proof}
(i) Let $\Omega \subset \mathbb T^N \times \mathbb{S}^1$ be an arbitrary nonempty open subset. In \cite[p.
335]{komornik-1992}, Komornik proves the estimate
  \begin{equation*}
  \|u\|_{L^2(\T^N \times \mathbb{S}^1)} \les \|u\|_{L^2(\Omega)} 
  \end{equation*}
for every function of the form
$$
  u(x,t) = \sum_{\alpha \in A}b_\alpha e^{i \alpha \cdot (x,t)}, \qquad \text{with}\; b_\alpha \in \C,\quad
  \sum_{\alpha \in A}|b_\alpha|^2 < \infty, 
$$
where
 $$
    A:= A_+ \cup A_- \qquad \text{with}\quad A_\pm := \Bigl\{ \bigl(d_1,\dots,d_N, \pm \sum_{j=1}^N d_j^2\bigr)\::\: d \in \Z^N\Bigr\} \; \subset \Z^{N+1}.
 $$
These functions $u$ are precisely the functions in $\ker(L_{\cA})= \ker \bigl(\partial_{tt}+ (-\Delta)^2)\bigr)$, so (i) follows.  \\
(ii) Suppose that $N=1$, i.e. $\T^N = \mathbb{S}^1$, and $\cA = (-\Delta)^m$ for some $m \ge 2$.
It suffices to prove the estimate
    \begin{equation}
    \label{eq-cp-space-time-komornik-extra}
  \|u\|_{L^2(\T^N \times \mathbb{S}^1)} \les \|u\|_{L^2(\Omega)} 
  \end{equation}
  for every function  
  \begin{equation}
    \label{eq:kernel-function-S-1}
u(x,t) = \sum_{\alpha \in A}b_\alpha e^{i \alpha \cdot (x,t)}, \qquad \text{with}\; b_\alpha \in \C,\quad \sum_{\alpha \in A}|b_\alpha|^2 < \infty,
  \end{equation}
where   
    $$
    A:= A_+ \cup A_- \qquad \text{with}\quad A_{\pm}:= \Bigl\{ \bigl(d, \pm |d|^m \bigr)\::\: d \in \Z
    \Bigr\} \;\subset \: \Z^2.
    $$
  As detailed in \cite{komornik-1992}, \eqref{eq-cp-space-time-komornik-extra} holds for such functions if
  $K(A)= 0$ where 
  $$
  K(A) = \inf \Bigl\{ \sum_{j=1}^k \frac{1}{d(A_j)} \::\: k \in \N,\: A_1,\dots,A_k \subset A,\: \bigcup_{j=1}^k A_j = A \Bigr\}
  $$
  and 
  $$
  d(B):= \inf \{|a-a'|\::\: a,a' \in B,\: a \not = a'\} \qquad \text{for a subset }B \subset \R^2.
  $$
  Note that $d(B)= \infty$ and therefore $\frac{1}{d(B)} = 0$ if $B$ consists of no more than one element.
  Therefore the value $K(A)$ does not change if finitely many points are removed or added to $A$. Moreover, by symmetry we have 
  $K(A_+)= K(A_-)$ and therefore
  $$
    K(A)\le K(A_+)+ K(A_-)= 2 K(A_+).
  $$
 Since, for every $n \in \N$, $A$ differs from $A^n_+:= \{\bigl(d, |d|^m \bigr) \in A \::\: |d| \ge n\}$ just
by finitely many points and therefore $K(A)=K(A^n)\leq 2K(A^n_+) \leq \frac{2}{d(A^n_+)}$, it suffices to show
that
  \begin{equation}
    \label{eq:sufficient-internal-distance}
    \lim_{n \to \infty} d(A^n_+)= \infty. 
  \end{equation}
  To see this, we note that for every two different points $(d,d^m), (e,e^m) \in A^n_+$ two cases may occur.
  If $|d| \not = |e|$, we have, since $m \ge 2$,  
  $$
|(d,|d|^m)-(e,|e|^m)| \ge \Bigl| |d|^m-|e|^m \Bigr| 
\ge (n+1)^m -n^m = \sum_{k=0}^{m-1}{m \choose k} n^k \ge m n
  $$
On the other hand, if $|d|=|e|$, we must have $d = -e$ and therefore
  $$
  |(d,d^m)-(e,e^m)| \ge |d-e| \ge 2n.
  $$
We thus conclude that (\ref{eq:sufficient-internal-distance}) and therefore (\ref{eq-cp-space-time-komornik-extra}) holds, and thus the proof of (ii) is finished.\\
To deduce (iii) for a nonempty open subset $\omega \subset M$, we note that every eigenfunction of $-\Delta$
on $\mathbb T^N$ with eigenvalue $\lambda \ge 0$ is of the form   
$$
    v(x) = \sum_{\beta \in B_\lambda} c_\beta e^{i \beta \cdot x} \qquad \text{with}\quad 
    B_\lambda:= \{\beta \in \Z^N \::\: \sum_{j=1}^N \beta_j^2 = \lambda\}
    $$
    and $c_\beta \in \C$ for $\beta \in B_\lambda$ with  $\sum_{\beta\in B_\lambda} |c_\beta|^2<\infty$.
    For such $v$, the function 
    $$
  u \in L^2(\T^N \times \mathbb{S}^1), \qquad 
  u(x,t) = \sum_{\beta \in B_\lambda}c_\beta e^{i \alpha_\beta \cdot (x,t)} \qquad \text{with }\:\alpha_\beta
 = (\beta, \lambda) \text{ for }\beta \in B_\lambda 
$$
belongs to $\ker(L_\cA u)$ for $\cA = (-\Delta)^2$. Applying (i) with $\Omega = \omega \times \mathbb{S}^1$
thus gives
\begin{align*}
  \|v\|_{L^2(M)}^2 
  &\sim \sum_{\beta \in B_\lambda} |c_\beta|^2 
  \sim \|u\|_{L^2(\T^N \times \mathbb{S}^1)}^2 
  \les \|u\|_{L^2(\Omega)}^2
   =\sum_{\beta,\beta' \in B_\lambda}c_\beta \overline{c_{\beta'}} \int_{\Omega}e^{i (\alpha_\beta-\alpha_{\beta'}) \cdot (x,t)}d(x,t)\\
  &=\sum_{\beta,\beta' \in B_\lambda}c_\beta \overline{c_{\beta'}} \int_{\omega}e^{i (\beta-\beta') \cdot x}
  \,dx \int_{\mathbb{S}^1} e^{i(\lambda^2-\lambda^2)t} dt =2 \pi \int_{\omega} \Bigl|\sum_{\beta\in B_\lambda} c_\beta
  e^{i \beta \cdot x} \Bigr|^2\,dx = 2\pi \|v\|_{L^2(\omega)}^2,
\end{align*}
so (\ref{eq:cp-space}) holds for $v$ with a constant independent of $v$. 
This proves the control property for $-\Delta$ on $\mathbb T^N$ and hence, by
Remark~\ref{rem-control-space-time}, the control property for $P(-\Delta)$ for $P\in\mathcal P_{-\Delta}$. So
Lemma~\ref{control-condition-implication} yields the control property of
$\omega\times\bS^1$ for $L_{\cA} = \partial_{tt}+P(-\Delta)$, which  
finishes the proof of (iii).
\end{proof}

\begin{cor}
  \label{cor-summary-CC-q}
  Let $(M,g)$ be a closed Riemannian manifold of dimension $N \ge 1$, and let $q \in L^\infty(M \times \mathbb{S}^1)$
  be a nonnegative function with $q>0$ on some nonempty open subset $\Omega \subset M \times \mathbb{S}^1$. Then the
  control condition $(CC)_q$ from the introduction is satisfied in the following cases:
  \begin{itemize}
  \item[(i)] $M = \mathbb{S}^N$, $\cA = P(-\Delta)$ for $P\in\mathcal P_{-\Delta}$ and $\Omega = \omega
  \times \mathbb{S}^1$ for some open subset $\omega \subset \mathbb{S}^N$ which intersects every great circle
  on $\mathbb{S}^N$.
  \item[(ii)] $M = \T^N$ and $\cA = (-\Delta)^2$.
  \item[(iii)] $M = \mathbb{S}^1$ and $\cA = (-\Delta)^m$ for $m \ge 2$.
  \item[(iv)] $M= \T^N$, $\cA=P(-\Delta)$ for $P\in\mathcal P_{-\Delta}$  and $\Omega = \omega \times
  \mathbb{S}^1$ for some nonempty open subset $\omega \subset \T^N$.
  \end{itemize}
\end{cor} 
\begin{proof}
  By Remark~\ref{rem-control-space-time}(i), it suffices to show that, in each of the cases $(i)-(iv)$, the
  open subset $\Omega$ has the control property for the operator $L_{\cA}$. In Case $(i)$, this follows from 
   Corollary~\ref{cor-CC-sphere} and Lemma~\ref{control-condition-implication}. In Case (ii), this follows
   from Theorem~\ref{thm-komornik}(i) thanks to $\T^1=\mathbb{S}^1$, and in Case (iii) it follows directly from
   Theorem~\ref{thm-komornik}(ii). Finally, Case (iv) is covered by Theorem~\ref{thm-komornik}(iii). 
\end{proof}

For the periodic wave equation on $\mathbb{S}^1\times \mathbb{S}^1$ we get a finer result. We consider functions $u:\mathbb{S}^1\times
\mathbb{S}^1\to\R$ as being $2\pi\Z^2$-periodic functions on $\R^2$. For a given set $\Omega\subset [0,2\pi)^2$  
define   
$$
  \tilde \Omega := \{(x,t+2\pi j) : (x,t)\in\Omega,\;j\in\{0,1\}\} \subset \R^2.
$$
So $\tilde \Omega$ is nothing but the union of $\Omega$ with its translate by $2\pi$ in the second
coordinate. For $\xi\in\R$ we define the sets 
$$  
  A_\xi:= \{x\in\R: (x,-x+\xi)\in\tilde \Omega\},\qquad
  B_\eta:= \{x\in\R: (x,x-\eta)\in\tilde \Omega\}.
$$

\begin{thm}\label{thm:CCq_for_wave1+1}
  The $q$-control property $(CC)_q$ holds for 
  $\cA = -\Delta, M= \mathbb{S}^1$ provided that $\inf_\Omega q>0$ for some subset $\Omega\subset \mathbb{S}^1\times
  \mathbb{S}^1$ such that 
  \begin{align*}
    \inf_{\xi\in [0,2\pi)} |A_\xi|>0
    \qquad\text{and}\qquad
    \inf_{\eta\in [0,2\pi)} |B_\eta|>0. 
  \end{align*}
\end{thm}
\begin{proof}
  Any $u\in E^0$ is $2\pi\Z^2$-periodic and a change of coordinates reveals $u(x,t)=\phi(x+t)+\psi(x-t)$ where
  $\phi,\psi$ are $2\pi$-periodic and square integrable. Then  
  \begin{align*}
    \int_{[0,2\pi)^2} q(x,t)|u|^2\,d(x,t)     
    &\geq \int_{\Omega} q(x,t)|u|^2\,d(x,t) \\
    &\geq \inf_{\Omega} q\cdot \|u\|_{L^2(\Omega)}^2  \\
    &=  \frac{1}{2} \inf_\Omega q \cdot \|u\|_{L^2(\tilde\Omega)}^2 \\
    &\geq   \frac{1}{4}\inf_\Omega q \cdot \int_{\tilde\Omega} |\phi(x+t)|^2+|\psi(x-t)|^2
    \,d(x,t) 
    \\
    &=   \frac{1}{4}\inf_\Omega q \cdot 
    \left(\int_{\R} \Big(\int_{A_\xi} 1 \,dx \Big)\cdot |\phi(\xi)|^2 \,d\xi  +
     \int_{\R} \Big(\int_{B_\eta} 1 \,dx \Big)\cdot |\psi(\eta)|^2 \,d\eta \right)
     \\
    &\geq   \frac{c}{4}\inf_\Omega q \cdot   \left(
    \int_0^{2\pi} |\phi(\xi)|^2\,d\xi + \int_0^{2\pi}|\psi(\eta)|^2\,d\eta  \right) 
    \\
    &=   \frac{c}{32\pi}\inf_\Omega q \cdot   \left(
    \int_0^{4\pi} \int_{-2\pi}^{2\pi} |\phi(\xi)|^2+ |\psi(\eta)|^2 \,d\xi \,d\eta  \right) 
    \\     
    &\geq c' \inf_\Omega q \cdot \int_{[0,2\pi)^2} |u(x,t)|^2 \,d(x,t). 
  \end{align*} 
  Here the last inequality follows from a change of coordinates $(\xi,\eta)\mapsto
  (x,t)=(\frac{\xi+\eta}{2},\frac{\xi-\eta}{2})$ that transforms $[0,4\pi]\times [0,2\pi]$ into a superset
  of $[0,2\pi)^2$.  This proves $(CC)_q$.  
\end{proof}

\begin{cor}
  The $q$-control property $(CC)_q$ holds for  $\cA = -\Delta, M= \mathbb{S}^1$ provided that $\inf_\Omega q>0$
  holds with 
  $$
     \Omega\supset  [a_1,b_1]\times [a_2,b_2]
     \quad\text{and } b_1+b_2-a_1-a_2 > 2\pi.
  $$ 
  In particular, $(CC)_q$ holds if $\Omega\supset \omega\times \mathbb{S}^1$ or
  $\Omega\supset \mathbb{S}^1\times \omega $ for some nonempty open set $\omega\subset \mathbb{S}^1$.
\end{cor}
\begin{proof}
  We use the criterion from the previous Theorem. We have 
  \begin{align*}
    (x,\xi-x)\in\tilde\Omega
    &\quad\Leftarrow\quad (x,\xi-x+2\pi j) \in [a_1,b_1]\times [a_2,b_2], j\in\{0,1\} \\
    &\quad\Leftrightarrow\quad x\in \bigcup_{j\in\{0,1\}} [\max\{a_1,\xi-b_2+2\pi
    j\},\min\{b_1,\xi-a_2+2\pi j\}].
  \end{align*}
  For all $\xi\in [0,2\pi)$ at least one of these intervals is nondegenerate in view of $b_1+b_2-a_1-a_2 >
  2\pi$. Indeed, the assumption implies that there is a small $\eps>0$ independent of $\xi$ such that 
  $\frac{\xi-a_1-a_2}{2\pi} -\eps$ and $\frac{\xi-b_1-b_2}{2\pi}+\eps$ have distance larger than $1$, i.e., 
  $$
    \frac{b_1+b_2-a_1-a_2}{2\pi} > 1 + 2\eps. 
  $$
  It is then possible to choose, for a given $\xi\in [0,2\pi)$, 
  $$
    j(\xi)\in\{0,1\}
     \quad\text{such that}\quad
     \frac{b_1+b_2-\xi}{2\pi} -\eps > j(\xi) > \frac{a_1+a_2-\xi}{2\pi}+\eps.  
  $$ 
  Note that the assumption implies that the upper bound is always smaller than 2 and the lower bound is always
  bigger than -1, so $j(\xi)\in\{0,1\}$ is not restrictive.  This proves  
  \begin{align*}
    \inf_{\xi\in [0,2\pi)} |A_\xi|
    &\geq \inf_{\xi\in [0,2\pi)} \Big| \bigcup_{j\in\{0,1\}} [\max\{a_1,\xi-b_2+2\pi
    j\},\min\{b_1,\xi-a_2+2\pi j\}]\Big| \\
    &\geq  \inf_{\xi\in [0,2\pi)} \big(\max\{a_1,\xi-b_2+2\pi j(\xi)\}-\min\{b_1,\xi-a_2+2\pi j(\xi)\}\big) \\
    &=  \inf_{\xi\in [0,2\pi)} \min\{b_1-a_1,b_1+b_2-\xi-2\pi j(\xi),\xi-a_1-a_2+2\pi j(\xi),b_2-a_2\} \\
    &\geq  \min\{b_1-a_1,2\pi \eps,b_2-a_2\}  >0
  \end{align*}
 Analogously, one proves $\inf_{\eta\in [0,2\pi)} |B_\eta|>0$,  
  and Theorem~\ref{thm:CCq_for_wave1+1} gives the claim.
\end{proof}

\section{Proofs of
Theorems~\ref{thm-wave-1+1},~\ref{thm-S-1-intro},~\ref{thm-T-N-intro},~\ref{thm-intro-sphere}
and~\ref{thm-intro-sphere-klein-gordon}}
\label{sec:proofs-theor-refthm}

 In this section, we collect all relevant material to prove  
 Theorems~\ref{thm-wave-1+1},~\ref{thm-S-1-intro},~\ref{thm-T-N-intro},~\ref{thm-intro-sphere} and
 \ref{thm-intro-sphere-klein-gordon} from the introduction. For this we first note that it is well known that
 the operators $\cA$ considered in these theorems satisfy assumption $(A)$. Hence, in order to deduce these
theorems from the abstract existence result given in Theorem~\ref{thm-general-hyperbolic-intro}, we only need
to show that the conditions $(CE)_p$ and $(CC)_q$ are also satisfied for suitable $2<p<p^*$ where the
nonnegative function $q\in L^\infty(M\times \mathbb{S}^1)$ and the critical exponent $p^*$ are chosen depending on the
operator. We refer to Table~\ref{tab:results} for an overview.  

\bigskip

In the case of Theorem~\ref{thm-wave-1+1}, the condition $(CE)_p$  is a consequence of
Corollary~\ref{cor-compact-embedding-m-even-S-N} and $(CC)_q$ holds by Theorem~\ref{thm:CCq_for_wave1+1}
thanks to assumption~\eqref{eq:q-con-wave-1+1}. \\[0.2cm]
In the case of Theorem~\ref{thm-S-1-intro}, we first note that
condition $(CE)_p$ is satisfied by Corollary~\ref{cor-compact-embedding-m-even-S-N}, applied in the case
$N=1$. Moreover, condition $(CC)_q$ follows from assumption~(\ref{eq:q-con-S1}) and
Corollary~\ref{cor-summary-CC-q}(iii) and (iv) in the case $N=1$, i.e., $\T^N = \mathbb{S}^1$. This completes the
proof of Theorem~\ref{thm-S-1-intro}.\\[0.2cm] 
In the case of Theorem~\ref{thm-T-N-intro}, the condition $(CE)_p$ is satisfied by
Corollary~\ref{cor-compact-embedding-m-even-S-N}, while condition $(CC)_q$ follows from
assumption~\eqref{eq:q-con-TN} and Corollary~\ref{cor-summary-CC-q}(ii) and (iv). in the case $N=1$, i.e.,
$\T^N = \mathbb{S}^1$. This completes the proof of Theorem~\ref{thm-T-N-intro}.\\[0.2cm] 
In the case of
Theorem~\ref{thm-intro-sphere}, the condition $(CE)_p$ is satisfied by Corollary~\ref{cor-compact-embedding-m-even-S-N},  while condition $(CC)_q$ follows from assumption~(\ref{eq:q-con-SN-intro}) and Corollary~\ref{cor-summary-CC-q}(i).
This completes the proof of Theorem~\ref{thm-intro-sphere}.\\[0.2cm]
In the case of Theorem~\ref{thm-intro-sphere-klein-gordon}, the condition $(CE)_p$ is satisfied by
Corollary~\ref{cor-compact-embedding-N-odd-S-N},  while condition $(CC)_q$ follows from
assumption~(\ref{eq:q-con-SN-intro-1}) and Corollary~\ref{cor-summary-CC-q}(i).
This completes the proof of Theorem~\ref{thm-intro-sphere-klein-gordon}.\\ 

{\bf Acknowledgement:} The authors would like to thank Lorenzo Giaretto for helpful comments on the first version of the paper.  

\bibliographystyle{abbrv}
\bibliography{biblio}

\end{document}